\documentclass{amsart}
\usepackage{amsthm,amssymb,latexsym,amsxtra}

\usepackage[cmtip,all]{xy}

\newcommand{\lto}{\longrightarrow}
\newcommand{\eto}{\hookrightarrow} 
\newcommand{\tto}{\twoheadrightarrow}

\DeclareMathOperator{\Spec}{Spec}

\DeclareMathOperator{\Ker}{Ker}
\DeclareMathOperator{\Imm}{Im}

\theoremstyle{plain}
\newtheorem{theorem}{Theorem}[section]
\newtheorem{corollary}[theorem]{Corollary}
\newtheorem{lemma}[theorem]{Lemma}
\newtheorem{proposition}[theorem]{Proposition}

\theoremstyle{definition}
\newtheorem{definition}[theorem]{Definition}
\newtheorem{remark}[theorem]{Remark}
\newtheorem{block}[theorem]{}

\theoremstyle{remark}
\newtheorem*{notation}{Notation and conventions}

\title[Hurwitz moduli varieties]{Hurwitz moduli varieties parameterizing 
pointed covers of an algebraic curve with a fixed monodromy group}
\author[V.~Kanev]{Vassil Kanev}
\date{October 20, 2025}
\address{V. Kanev, Institute of Mathematics and Informatics, 
         Bulgarian Academy of Sciences, 8 Acad. Georgi Bonchev Str.
         1113 Sofia (BULGARIA).}
\email{kanev@math.bas.bg}
                                                                               
\thanks{This version of the article has been accepted for publication, after peer review,  
in Atti Accad. Naz. Lincei Rend. Lincei Mat. Appl., https://doi.org/10.4171/RLM/1061,\\
\copyright \, 2025 Accademia Nazionale dei Lincei.}

\begin{document}

\begin{abstract}
Given a smooth, projective curve $Y$, a point $y_0 \in Y$, a positive integer 
$n$, and a transitive subgroup $G$ of the symmetric group  $S_{d}$  we  study 
smooth, proper families, parameterized by algebraic varieties, of pointed \linebreak
degree $d$ covers of $(Y,y_0)$,  $(X,x_{0})\to  (Y,y_0)$,  branched  in  $n$ 
points of $Y\setminus y_{0}$, whose monodromy group equals $G$. We construct 
a Hurwitz space $H$, an algebraic variety whose points are in bijective correspondence with 
the equivalence classes of pointed covers of $(Y,y_0)$ of this type. 
We construct explicitly a family parameterized by $H$, whose fibers belong to the 
corresponding equivalence classes, and prove  that  it  is 
universal. We use classical tools  of  algebraic  topology  and  of  complex 
algebraic geometry.
\end{abstract}

\maketitle

\section{Introduction}\label{s1}
Fulton studied in \cite{Fu} smooth, proper families of covers of $\mathbb{P}^{1}$ of 
degree \linebreak
$d\geq 3$ simply branched in $n\geq 2d-2$  points,  parameterized  by 
schemes over $\mathbb{Z}$. In particular, over $\mathbb{C}$ he constructed a 
universal  family  of  such  covers,  parameterized  by  the  Hurwitz  space 
$H^{d,n}$, a variety whose points correspond bijectively to the  equivalence 
classes of covers of $\mathbb{P}^{1}$ of the above type.
\par
This paper is concerned with generalizing Fulton's 
results to degree $d$ covers of an arbitrary smooth, irreducible, projective 
curve $Y$, branched  in  $n$  points,  whose  monodromy  group  is  a  fixed 
transitive subgroup $G$ of $S_{d}$. When $Y=\mathbb{P}^{1}$ this problem was 
first studied by Fried in \cite{Fr} in connection with the  Inverse Galois Problem. 
An obstruction to the existence of a universal family  of  covers 
of $Y$ of this type is the presence of non trivial covering automorphisms of 
$f:X\to Y$. This happens if the centralizer of $G$ in $S_{d}$ is nontrivial. 
This obstacle may be avoided by working with pointed  covers  of  $(Y,y_0)$, 
where $y_{0}$ is a marked point of $Y$. Namely, we consider pairs 
$(f:X\to Y,x_0)$ where $X$ is a smooth, irreducible, projective  curve,  $f$ 
is a cover branched in a set  $D\subset Y$ of $n$ points, $y_{0}\notin D$, 
$x_{0}\in f^{-1}(y_{0})$, and there is a numeration of $f^{-1}(y_{0})$  such 
that the monodromy group  obtained  from  the  action  of 
$\pi_1(Y\setminus D,y_0)$ on $f^{-1}(y_0)$ equals $G$.
We construct an algebraic variety $H$, a Hurwitz space, whose points 
correspond bijectively to the equivalence classes of pointed covers of 
$(Y,y_0)$ of the above type.
\par
We study smooth, proper families of 
pointed covers of $(Y,y_0)$ parameterized by algebraic varieties.  Our  main 
result is the construction of a  universal  family  of  degree  $d$  pointed 
covers of $(Y,y_0)$, branched in $n$ points, with monodromy group $G\subset 
S_{d}$, parameterized by the Hurwitz space $H$. The construction is  explicit 
and uses the universal family of pointed Galois covers 
of $(Y,y_0)$ with Galois group isomorphic to $G$ constructed in \cite{K7}.
\par
The covers $f:X\to Y$ with restricted monodromy group and the related Galois 
covers $p:C\to Y$ yield polarized abelian varieties isogenous to abelian 
subvarieties of the Jacobian  variety  $J(X)$  \cite{Do, K5, Me, CLRR}.  
Given  a  smooth, 
proper family of such covers one obtains a morphism of its parameter variety 
to a certain moduli space of  polarized  abelian  varieties,  which  may  be 
studied by means of the associated variation  of  the  Hodge  structures  of 
weight one \cite{Gri1, Gri2, CMP}.  This approach to the 
moduli of polarized abelian varieties was used in \cite{K3, K4} where the 
unirationality of $\mathcal{A}_{3}(1,1,d)$ and $\mathcal{A}_{3}(1,d,d)$ with 
$2\leq d\leq 4$ was proved by means of families of simply ramified covers of 
elliptic curves branched in 6 points. The Prym--Tyurin morphism from a 
Hurwitz space to $\mathcal{A}_{6}$ was used in \cite{Al} to prove that every 
sufficiently general principally polarized abelian variety of dimension 6 is 
isomorphic to a Prym--Tyurin variety of a cover of $\mathbb{P}^{1}$ of degree 
27, branched in 24 points, with monodromy group the Weyl group 
$W(E_{6})\subset S_{27}$. 
\par
Emsalem studied in \cite{Em1} the moduli properties of the Hurwitz spaces  of 
pointed degree $d$ covers of $(Y,y_0)$,  where  $Y=\mathbb{P}^{1}$,  with  a 
fixed monodromy group $G\subset  S_{d}$.  He  worked  with  families,  
which are not proper, whose 
fibers are \'{e}tale covers of open subsets of $\mathbb{P}^{1}$, complements 
of $\leq n$ points, and whose parameter varieties are  \'{e}tale  covers  of 
open   subsets   of   $(\mathbb{P}^{1}\setminus   y_{0})_{\ast}^{(n)}$.   He 
constructed  universal  families  of  this  type   parameterized   by   the 
corresponding Hurwitz spaces \cite[Th\'{e}or\`{e}me~$3'$]{Em1}.
\par
Section~2 contains  generalities  on  degree  $d$ 
covers of a curve $Y$ with a fixed monodromy group $G\subset S_{d}$. We 
associate with every cover a data called monodromy invariant by which 
one parameterizes the equivalence classes of covers of $Y$ of the above type.
\par
Section~3 of the paper  contains generalities on smooth, proper families of degree  $d$ 
covers of a curve $Y$ with a fixed monodromy group $G\subset S_{d}$. We 
discuss their connection with the smooth, proper families of Galois covers 
of $Y$ with Galois group isomorphic to $G$.
\par
In  Section~4,  given  a  curve  $Y$,  a 
point $y_{0}\in Y$, a positive integer $n$, and a transitive subgroup 
$G\subset S_{d}$, we endow with a structure of an algebraic variety
the set $H$ which parameterizes the equivalence classes of 
pointed degree $d$ covers of $(Y,y_0)$, branched in $n$ points, with monodromy group $G$.
We construct explicitly in Theorem~\ref{3.19} a smooth, proper family 
of pointed covers of $(Y,y_0)$ of this type \linebreak
$(\phi:\mathcal{X}\to Y\times H,\xi:H\to \mathcal{X})$, parameterized by $H$,
whose fiber over every $h\in H$ belongs to the equivalence class of covers
corresponding to $h$.
We give the explicit form of $\phi$ locally 
at the ramification points in  analytic  coordinates  (Proposition~\ref{3.8} 
and Remark~\ref{3.20a}). We prove that the constructed family is universal in 
Theorem~\ref{3.25}. We mention that a key ingredient of the proof is the 
use of the criterion for extending morphisms from \cite{K2}.
We prove in Proposition~\ref{3.49} that, after performing an \'{e}tale base change, every 
smooth, proper family of pointed degree $d$ covers of $(Y,y_0)$ with monodromy group 
$G\subset S_{d}$ is isomorphic to the quotient by the isotropy 
subgroup $H=G(1)$, of a smooth, proper family of pointed 
$G$-covers of $(Y,y_0)$. We give in Theorem~\ref{3.34} a variant of the main 
theorem in which the pointed covers of $(Y,y_0)$ have local monodromies at 
the branch points in a set of marked conjugacy classes of $G$. 

\begin{notation}
We assume the base field is $\mathbb{C}$. Algebraic varieties are reduced,  separated, 
possibly reducible schemes of finite type, \emph{points} are closed  points. 
Fiber  products  and 
pullbacks are those defined in the category of schemes over $\mathbb{C}$. A 
cover $f:X\to Z$ of algebraic varieties is  a  finite,  surjective 
morphism. If $Y$ is a smooth, projective, irreducible curve we denote by
$Y^{(n)}$ the $n$-th symmetric power of $Y$ and by 
$Y^{(n)}_{\ast}$ the subset of $Y^{(n)}$ consisting of effective divisors of $Y$ of degree $n$ without 
multiple points. We identify such divisors with their support. If $\Omega \subset Y^{(n)}$,
then $\Omega_{\ast} = \Omega \cap Y^{(n)}_{\ast}$.
Given an algebraic variety $(X,\mathcal{O}_{X})$  
the  canonically 
associated complex space is denoted by $(X^{an},\mathcal{O}_{X^{an}})$ 
\cite{SGA1}. Its  topological  space  is  denoted  by  $|X^{an}|$. 
We assume that the homotopy of paths leaves the end points fixed 
\cite[Chapter~2, \S2]{Mas}. The set of paths homotopic to $\alpha$ is 
denoted by $[\alpha]$. 
The product of the paths $\alpha$ and $\beta$ is denoted  by  $\alpha  \cdot 
\beta$ and it equals the path $\gamma$, where $\gamma(t)=\alpha(2t)$  if 
$t\in [0,\frac{1}{2}]$, $\gamma(t)=\beta(2t-1)$ if  $t\in  [\frac{1}{2},1]$. 
Given a covering space $p:M\to N$ of the topological  space  $N$,  the 
map $p$ is called topological covering map. Lifting a path $\alpha$ of $N$ from
an initial point $z\in M$ the terminal point is denoted by $z\alpha$.
\end{notation}

\section{Parameterization of covers with a fixed monodromy group}\label{s2}
\begin{block}\label{2.0}
Throughout the paper $Y$ is a smooth, projective, irreducible curve of genus
$g\geq 0$, $n$ is a positive integer, $G$ is a finite  group,  which  can  be 
generated by $2g+n-1$ elements, $\Lambda$ is a finite set, $|\Lambda|=d$,  on 
which $G$ acts on the right faithfully and transitively, we choose an element
$\lambda_0\in \Lambda$  as  a 
marked element, $G(\lambda_0)=\{g\in G| \lambda_0g=\lambda_0\}$. 
We identify $G$ with its  image  in  the  symmetric  group  $S(\Lambda)\cong 
S_{d}$, where $S(\Lambda)$ acts on the right.
\end{block}
\begin{definition}\label{2.0a}
We call a finite, surjective morphism $f:X\to Y$ a \emph{cover}  of  $Y$  if 
$X$ is a smooth, projective, irreducible curve. We say that two covers 
$f:X\to Y$ and $f_{1}:X_{1}\to Y$ are equivalent if there exists a  
covering isomorphism $h:X\to X_{1}$, i.e. one such that $f_{1}\circ h=f$. 
A $G$-cover of $Y$ is a cover $p:C\to  Y$ as above, such that 
$G$ acts faithfully on  the  left  on  $C$  by 
automorphisms of $C$, $p$ is $G$-invariant and $\overline{p}:C/G \to Y$ is an isomorphism.
Two $G$-covers of $Y$, $p:C\to  Y$  and  $p_{1}:C_{1}\to  Y$  are  called 
$G$-equivalent  if  there  exists a covering isomorphism $h:C\to C_{1}$ 
which is $G$-equivariant.
\end{definition}
\begin{block}\label{2.0aa}
A necessary condition for two covers $f:X\to Y$ and  $f_{1}:X_{1}\to  Y$  to 
be equivalent is that they have the same branch locus $D$. 
Let $y_0\in Y\setminus D$. 
Let $Y^{an}$ be  the 
Riemann surface associated with  $Y$.  Then  $\pi_1(Y^{an}\setminus  D,y_0)$ 
acts on the right on $f^{-1}(y_0)$ by $x\cdot [\alpha] = x\alpha$ for 
$\forall x\in f^{-1}(y_0)$ and $\forall [\alpha] \in 
\pi_1(Y^{an}\setminus D,y_0)$, and similarly it acts on the right on 
$f_{1}^{-1}(y_0)$. Then $f$ is equivalent to $f_{1}$ if and  only  if  there 
exists a $\pi_1(Y^{an}\setminus D,y_0)$-equivariant bijection 
$\mu: f^{-1}(y_0)\to f_{1}^{-1}(y_0)$.  In  fact  let 
$X'=f^{-1}(Y\setminus D)$, $X_{1}'=f_{1}^{-1}(Y\setminus D)$.  Then  such  a 
bijection yields a covering biholomorphic map $X'^{an}\to X'^{an}_{1}$ over
$Y\setminus D$, which extends to a covering biholomorphic map 
$X^{an}\to X^{an}_{1}$ over $Y^{an}$. Furthermore this map equals  $h^{an}$, 
where $h:X\to X_{1}$ is a covering isomorphism over $Y$. Abusing notation we 
will write $\pi_1(Y\setminus D,y_0)$ instead of 
$\pi_1(Y^{an}\setminus D,y_0)$.
\end{block}
\begin{definition}\label{2.1}
Let $y_0\in Y$. Let $f:X\to Y$ be a cover such that $y_{0}$ does not belong 
to the branch locus of $f$ and let 
$x_{0}\in f^{-1}(y_0)$. We call $(f:X\to Y,x_{0})$ a 
\emph{pointed cover} of $(Y,y_0)$. Two pointed covers 
$(f:X\to Y,x_{0})$ and $(f_{1}:X_{1}\to Y,x'_{0})$ of $(Y,y_0)$  are  called 
equivalent if there exists an isomorphism $h:X\to X_{1}$ such that 
$f_{1}\circ h=f$ and $h(x_{0})=x'_{0}$.
A pointed $G$-cover of $(Y,y_0)$  is  a  pointed cover  
$(p:C\to Y,z_0)$ as above such that $p:C\to Y$ is a  $G$-cover. 
Two pointed $G$-covers of $(Y,y_0)$, $(p:C\to  Y,z_0)$  and  
$(p_{1}:C_{1}\to  Y,z'_0)$  are  called 
$G$-equivalent if 
there  exists  a  $G$-equivariant   isomorphism   $h:C\to C_{1}$ 
such that  $p=p_{1}\circ  h$ and $h(z_0)=z'_0$.
\end{definition}
\begin{block}\label{2.1a}
The isomorphism $h$ of Definition~\ref{2.1} is unique.
In fact, if $h_{1}$ is a second 
one,  then  $\varphi   =   h_{1}^{-1}\circ   h:X\to   X$   is   a   covering 
automorphism over $Y$, such that  $\varphi(x_{0})=x_{0}$.  Let  $D$  be  the 
branch locus of $f$ and let $X'=X\setminus f^{-1}(D)$. Then 
$f^{an}|_{X'^{an}}:X'^{an}\to Y^{an}\setminus D$ is a topological covering 
map, $X'^{an}$ is connected, so $\varphi^{an}|_{X'^{an}}$ is  the  identity  map, 
since $\varphi^{an}(x_{0})=x_{0}$. Therefore $h=h_{1}$.
\end{block}
\begin{block}\label{2.2}
Let $D=\{b_1,\ldots ,b_n\}\subset Y\setminus y_0$. Let 
$\overline{U}_{1},\ldots,\overline{U}_{n}$ be embedded closed disks in 
$Y\setminus y_0$ which are disjoint and such  that  $b_{i}\in U_{i}$  for  
$\forall  i$, where 
$U_{i}$ is the interior of $\overline{U}_{i}$.
For every $i=1,\ldots,n$ let us choose a path 
$\eta_{i}:I\to  Y\setminus \cup_{j=1}^{n} U_{j}$  
such  that  $\eta_{i}(0)=y_{0}$, 
$\eta_{i}(1)\in  \partial   \overline{U}_{i}$   and   let   
$\gamma_{i}:I\to Y\setminus D$ be the closed path which  starts  at  
$y_{0}$,  travels  along $\eta_{i}$,   then   makes   a   counterclockwise   
loop   along   
$\partial \overline{U}_{i}$ and returns back to  $y_{0}$  along  
$\eta_{i}^{-}$. Let $f:X\to Y$ be a $d$-sheeted cover of $Y$ unbranched
at $y_{0}$. Let $X'=X\setminus f^{-1}(D)$, $f'=f|_{X'}$. The branch locus 
of $f$ equals $D$ if and only if $f':X'\to Y\setminus D$ is  unramified  and 
the monodromy $m:\pi_1(Y\setminus D,y_0)\to S(f^{-1}(y_0))$ defined by
$x\cdot m([\alpha]) = x\alpha$ for $\forall x\in f^{-1}(y_0)$ and 
$\forall [\alpha]\in \pi_1(Y\setminus D,y_0)$ satisfies the condition
\begin{equation}\label{e2.2}
m([\gamma_{1}])\neq 1, \ldots , m([\gamma_{n}])\neq 1.
\end{equation}
Let $m:\pi_1(Y\setminus D,y_0)\to S_d$ be a homomorphism and let 
$G=\Imm(m)$.  Suppose  that  $m$  satisfies  Condition~\eqref{e2.2}.  If  one 
chooses in a different way $\overline{U}_{1},\ldots,\overline{U}_{n}$ 
and $\eta_{1},\ldots,\eta_{n}$ so as to satisfy the above conditions, then
for every  $i\in  [1,n]$  the  new  element  $[\gamma'_{i}]$  belongs  to  the 
conjugacy class of $[\gamma_{i}]$ in $\pi_1(Y\setminus D,y_{0})$, so 
$m([\gamma'_{i}])$ belongs to the conjugacy class of $m([\gamma_{i}])$ in
$G$. Therefore Condition~\eqref{e2.2} for $m$ is independent of  the  choice 
of $\overline{U}_{1},\ldots,\overline{U}_{n}$ and that of 
$\eta_{1},\ldots,\eta_{n}$.
\end{block}
\begin{lemma}\label{2.2b}
Let $G\subset S(\Lambda)$ be as in \S\ref{2.0}. Let 
$N_{S(\Lambda)}(G) = \{\sigma \in S(\Lambda)|\sigma G\sigma^{-1}=G\}$. 
Let $f:X\to Y$ and $f_{1}:X_{1}\to Y$ be two covers  with  the  same  branch 
locus $D$. Let $y_0\in Y\setminus D$. Suppose that there are bijections 
$\varepsilon : \Lambda \to f^{-1}(y_0)$, 
$\varepsilon_{1} : \Lambda \to f_{1}^{-1}(y_0)$ and epimorphisms 
$m:\pi_1(Y\setminus D,y_0)\to G$, $m_{1}:\pi_1(Y\setminus D,y_0)\to G$ such
that $\varepsilon (\lambda m([\alpha])) = \varepsilon (\lambda )\alpha$ and 
$\varepsilon_{1} (\lambda m_{1}([\alpha])) = \varepsilon_{1} (\lambda )\alpha$ 
for $\forall \lambda \in \Lambda$ and 
$\forall [\alpha] \in \pi_1(Y\setminus D,y_0)$. Then $f$ is equivalent to 
$f_{1}$ if and only if there exists a $\sigma  \in  N_{S(\Lambda)}(G)$  such 
that $m_{1}=\sigma m \sigma^{-1}$. The covering isomorphism 
$h:X\to X_{1}$ is unique if and only if the centralizer of $G$ in $S(\Lambda)$
is trivial, $Z_{S(\Lambda)}(G) = \{1\}$.
\end{lemma}
\begin{proof}
The covers $f:X\to Y$ and $f_{1}:X_{1}\to Y$ are equivalent if and  only  if 
there exists a $\pi_1(Y\setminus D,y_0)-$equivariant bijection 
$\mu : f^{-1}(y_0)\to f_{1}^{-1}(y_0)$ (cf. \S\ref{2.0aa}).  Let  $\mu$  be 
such a bijection. Let $s:\Lambda \to \Lambda$ be the  bijection,  which  makes 
the following diagram commutative:
\begin{equation}\label{e2.5}
\xymatrix{
\Lambda\ar[r]^-{\varepsilon}\ar[d]_-{s}&f^{-1}(y_0)\ar[d]^-{\mu}\\
\Lambda\ar[r]^-{\varepsilon_1}&f_{1}^{-1}(y_0).
}
\end{equation}
Let $\sigma \in S(\Lambda)$ satisfy $s(\lambda)=\lambda \sigma^{-1}$ for 
$\forall \lambda \in \Lambda$. Let $\lambda \in \Lambda$ and let 
$x=\varepsilon (\lambda)\in f^{-1}(y_0)$. For 
$\forall [\alpha]\in \pi_1(Y\setminus D,y_0)$ one has 
$\mu(x\alpha)=\mu(x)\alpha$ and
\[
\begin{split}
&\mu(x\alpha) = \mu(\varepsilon (\lambda)\alpha)
= \mu(\varepsilon(\lambda m([\alpha]))) = 
\varepsilon_{1}(\lambda m([\alpha])\sigma^{-1})\\
&\mu(x)\alpha = \mu(\varepsilon (\lambda))\alpha
= \varepsilon_{1}(\lambda \sigma^{-1})\alpha = 
\varepsilon_{1}(\lambda \sigma^{-1} m_{1}([\alpha])).
\end{split}
\]
This shows that $m_{1}([\alpha])=\sigma m([\alpha])\sigma^{-1}$ for 
$\forall [\alpha] \in \pi_1(Y\setminus D,y_0)$. In particular 
$\sigma \in N_{S(\Lambda)}(G)$, since $m$ and $m_{1}$ are  epimorphisms.  We 
conclude that $m_{1}=\sigma m \sigma^{-1}$. Viceversa, suppose that 
$m_{1}=\sigma m \sigma^{-1}$ for some $\sigma \in N_{S(\Lambda)}(G)$. Let 
$s:\Lambda \to \Lambda$ be the bijection 
$\varepsilon(\lambda)=\lambda \sigma^{-1}$ and let 
$\mu:f^{-1}(y_0)\to  f_{1}^{-1}(y_0)$  be  the  bijection  which  makes  the 
diagram \eqref{e2.5} commutative. Then $\mu(x\alpha)=\mu(x)\alpha$ for 
$\forall x\in f^{-1}(y_0)$ and
$\forall [\alpha]\in \pi_1(Y\setminus D,y_0)$, therefore 
$f:X\to Y$ is equivalent to $f_{1}:X\to Y$.
\par
Suppose $f$ is equivalent to $f_{1}$. The covering isomorphism 
$h:X\to X_{1}$
is unique if and only if every covering automorphism 
$\varphi :X\to X$ equals the identity.  Consider  the  diagram  \eqref{e2.5} 
with $f_{1}=f$, $\varepsilon_{1}=\varepsilon$ and $m_{1}=m$. Let 
$\sigma \in S(\Lambda)$ satisfy 
$\mu(\varepsilon(\lambda))=\varepsilon (\lambda \sigma^{-1})$. Then 
$\mu(x\alpha)=\mu(x)\alpha$ for $\forall x\in f^{-1}(y_0)$ and 
$\forall [\alpha]\in \pi_1(Y\setminus D,y_0)$ if and only if 
$\sigma m([\alpha])\sigma^{-1} = m([\alpha])$ for 
$\forall [\alpha]\in \pi_1(Y\setminus D,y_0)$. Since $\Imm(m)=G$  this  holds 
if and only if $\sigma \in Z_{S(\Lambda)}(G)$. Therefore $\mu = id$, 
equivalently
$\varphi = id_X$,  if  and 
only if $Z_{S(\Lambda)}(G)=\{1\}$.
\end{proof}
\begin{proposition}\label{2.3}
In the setup of \S\ref{2.0} let $y_0\in Y$, let
$D=\{b_1,\ldots,b_n\}\subset Y\setminus y_0$ and let 
$(f:X\to Y,x_0)$ be a pointed cover of $(Y,y_0)$ of degree $d$
branched in $D$. The following conditions are equivalent:
\begin{enumerate}
\item
There  is  a  bijection  $\varepsilon  :\Lambda  \to  f^{-1}(y_0)$  and
$m:\pi_1(Y\setminus D,y_0)\to G$,   
an epimorphism,  such that
\begin{equation}\label{e2.3}
\begin{split}
\varepsilon (\lambda m([\alpha])) &= \varepsilon (\lambda )\alpha
\quad \text{for}\quad \forall \lambda \in \Lambda,
\forall [\alpha]\in \pi_1(Y\setminus D,y_0)\\
\varepsilon(\lambda_{0}) &=x_{0}.
\end{split}
\end{equation}
\item
Let $X'=X\setminus f^{-1}(D)$, $f'=f|_{X'}$ and let 
$\Gamma_{x_0}=f'_{\ast}\pi_1(X',x_{0})$. There is an epimorphism 
$m:\pi_1(Y\setminus D,y_0)\to G$ such that 
$\Gamma_{x_{0}}=m^{-1}(G(\lambda_0))$ where $G(\lambda_0)$ is the 
isotropy group of $\lambda_{0}$.
\end{enumerate}
Let $N(\lambda_0)=\{\sigma \in N_{S(\Lambda)}(G)|\lambda_{0}\sigma 
=\lambda_{0}\}$. Let $(f:X\to Y,x_0)$ and $(f_{1}:X_{1}\to Y,x'_0)$  be  two 
pointed covers of $(Y,y_0)$ branched in $D$ which satisfy Condition~(i) with 
$(\varepsilon,m)$  and  $(\varepsilon_{1},m_{1})$  respectively.  They   are 
equivalent if and only if there exists a $\sigma \in N(\lambda_0)$ such that 
$m_{1}=\sigma m \sigma^{-1}$. Furthermore  $\sigma  \in  N(\lambda_0)$  with 
this property is unique.
\end{proposition}
\begin{proof}
The conditions (i) and (ii) are equivalent since the map 
$\pi_1(Y\setminus D,y_0)\to f^{-1}(y_0)$ defined by 
$[\alpha]\mapsto x_{0}\alpha$ induces a 
$\pi_1(Y\setminus D,y_0)$-equivariant bijection between the set
of right cosets
$_{\Gamma_{x_{0}}}\backslash \pi_1(Y\setminus D,y_0)$ and 
$f^{-1}(y_0)$, $m$ induces a $\pi_1(Y\setminus D,y_0)$-equivariant 
bijection between 
$_{\Gamma_{x_{0}}}\backslash \pi_1(Y\setminus D,y_0)$ and 
$_{G(\lambda_0)}\backslash G \cong \Lambda$. Under these bijections 
$\lambda_{0}\mapsto G(\lambda_0)\mapsto \Gamma_{\lambda_{0}}\mapsto x_0$.
\par
The pointed  covers  $(f:X\to  Y,x_0)$  and  $(f_{1}:X_{1}\to  Y,x'_0)$  are 
equivalent   if   and   only   if   there   exists    a 
$\pi_1(Y\setminus D,y_0)$-equivariant bijection 
$\mu:f^{-1}(y_0)\to f_{1}^{-1}(y_0)$ such that $\mu(x_0)=x'_0$. By 
Lemma~\ref{2.2b} this is equivalent to the existence of a
$\sigma \in N_{S(\Lambda)}(G)$ such that 
$m_{1}=\sigma m \sigma^{-1}$ and $\lambda_{0}\sigma^{-1}=\lambda_{0}$,
 i.e. $\sigma \in N(\lambda_0)$. In order to prove the uniqueness of 
such a $\sigma$ it suffices to consider the case $m_{1}=m$ and prove that if 
$\sigma \in N(\lambda_0)$ satisfies $m=\sigma m \sigma^{-1}$, then 
$\sigma = id_{\Lambda}$. In fact, let $g\in G = \Imm(m)$, then 
$g=\sigma g \sigma^{-1}$ and 
$\lambda_{0}g=\lambda_{0}(\sigma g\sigma^{-1}) = 
(\lambda_{0}g)\sigma^{-1}$. Since $\Lambda =\lambda_{0}G$ this shows that 
$\lambda \sigma =\lambda$ for $\forall \lambda \in \Lambda$. 
\end{proof}
\begin{block}\label{2.5b}
The set 
\begin{equation}\label{e2.5b}
\begin{split}
H^G_n(Y,y_0) = \{&(D,m)| D\in (Y\setminus y_0)^{(n)}_{\ast}, 
m:\pi_1(Y\setminus D,y_0)\to G \\
& \text{is an epimorphism which satisfies Condition~\eqref{e2.2}} 
\}
\end{split}
\end{equation}
is in bijective correspondence with the set of  $G$-equivalence  classes  of 
pointed $G$-covers of $(Y,y_0)$ branched in $n$ points $\{[p:C\to Y,z_{0}]\}$ 
(cf. \cite{K1,K7}). Given a pointed $G$-cover  
$(p:C\to  Y,z_0)$ of $(Y,y_0)$,  
 $p(z_{0})=y_{0}$, the associated pair $(D,m)$, its monodromy invariant,  
consists  of  the  branch  locus  $D$ of $p$ and of the 
homomorphism \linebreak
$m:\pi_1(Y\setminus D,y_0)\to G$ defined as follows. Let 
$C'=p^{-1}(Y\setminus D)$. For every loop $\alpha :I\to Y\setminus D$  based 
at $y_{0}$ let $z_{0}\alpha$ be the end point of the lifting of $\alpha$  in 
$C'$ with initial point $z_{0}$. Let $z_{0}\alpha = gz_{0}, g\in G$.  Then 
$m([\alpha])=g$.
\end{block}
\begin{definition}\label{2.6}
In the setup of \S\ref{2.0} let 
$N(\lambda_0)=\{\sigma   \in    N_{S(\Lambda)}(G)|    \lambda_{0}\sigma    = 
\lambda_{0}\}$. Consider the left action of $N(\lambda_0)$ on 
$H^G_n(Y,y_0)$ defined by $\sigma \ast (D,m) = (D,\sigma m \sigma^{-1})$.  We 
denote by 
$H^{\Lambda,G}_{n,\lambda_0}(Y,y_0)$ the quotient set
\begin{equation*}
H^{\Lambda,G}_{n,\lambda_0}(Y,y_0) = H^G_n(Y,y_0)/N(\lambda_0).
\end{equation*}
\end{definition}
\begin{definition}\label{2.6a}
Let $y_0\in  Y$.  A  pointed  cover  $(f:X\to  Y,x_0)$  of  $(Y,y_0)$  which 
satisfies the conditions of Proposition~\ref{2.3}(i) for some 
$D\in (Y\setminus y_0)^{(n)}_{\ast}$ and a pair $(\varepsilon,m)$ is called 
\emph{pointed}  $(\Lambda,G)$-\emph{cover}  of  $(Y,y_0)$  branched  in  $n$ 
points. The pair
\begin{equation*}
(D,m^{N(\lambda_0)}) = (D,\{\sigma m \sigma^{-1}\}_{\sigma \in N(\lambda_0)})
\in H^{\Lambda,G}_{n,\lambda_0}(Y,y_0)
\end{equation*}
is called its \emph{monodromy invariant}.
\end{definition}
\begin{lemma}\label{2.6b}
Let $(p:C\to Y,z_0)$ be a pointed $G$-cover of 
$(Y,y_0)$ with monodromy invariant $(D,m)$ (cf. \S\ref{2.5b}).
Consider the 
action of $G$ on $\Lambda \times C$ defined by 
$g(\lambda,z)=(\lambda g^{-1},gz)$. Let
$X = (\Lambda \times C)/G := \Lambda \times^{G}C$. 
Let $\pi : \Lambda \times C \to X$ be the quotient morphism, let 
$f:X\to Y$ be the morphism defined by $f(\pi(\lambda,z))=p(z)$ and let 
$x_{0}=\pi(\lambda_{0},z_{0})$.
Then $(f:X\to Y,x_0)$ is a 
pointed $(\Lambda,G)$-cover of $(Y,y_0)$ with monodromy invariant 
$(D,m^{N(\lambda_0)})$. 
\end{lemma}
\begin{proof}
Let $G(\lambda_{0})=H$. Let $\rho:C\to C/H$ be the quotient morphism. One has 
$\Lambda\times^{G}C\cong C/H$. In fact, let us choose for every 
$\lambda \in \Lambda$ an element $a_{\lambda} \in G$ such that 
$\lambda = \lambda_0a_{\lambda}$, let $a_{\lambda_0}=1$. The morphisms 
$C\to  \Lambda  \times^{G}C$  and  $\Lambda  \times  C  \to   C/H$,   defined by 
\begin{equation}\label{e2.7a}
z \mapsto \pi(\lambda_0,z) \quad \text{and}\quad (\lambda,z)\mapsto 
\rho(a_{\lambda}z),
\end{equation}
are respectively $H$-invariant and $G$-invariant. The induced morphisms 
$G/H \to \Lambda \times^{G}C$ and $\Lambda \times^{G}C\to G/H$  are  inverse 
to each other. The curve $C$ is projective, smooth and irreducible, so 
$X \cong C/H$ has the same properties.
Consider the map
\begin{equation}\label{e2.8}
\varepsilon :\Lambda \to f^{-1}(y_0)\quad \text{defined by}\quad 
\varepsilon(\lambda)=\pi(\lambda,z_{0}).
\end{equation}
It is bijective since $G$ acts on $p^{-1}(y_0)$ transitively  and freely. 
One has $\varepsilon(\lambda_{0})=x_{0}$. Let 
$\varepsilon(\lambda)=\pi(\lambda,z_{0})$   be   an   arbitrary   point   of 
$f^{-1}(y_0)$. Let $\alpha :I\to Y\setminus D$ be a loop  based  at  $y_{0}$ 
and let $g=m([\alpha])$. Let $C'=p^{-1}(Y\setminus D)$. Lifting $\alpha$  in 
$\Lambda \times C$ with initial point $(\lambda,z_{0})$ the end point is 
$(\lambda,z_{0}\alpha)=(\lambda,gz_{0})$. 
Applying $\pi$ one obtains
\begin{equation*}
\varepsilon(\lambda)\alpha = \pi(\lambda,gz_{0}) =
\pi(\lambda g,z_{0}) = \varepsilon(\lambda m([\alpha])). 
\end{equation*}
Let $y\in Y\setminus D$ and let $z\in p^{-1}(y)$. The map 
$\Lambda  \to  f^{-1}(y)$  defined  by  $\lambda  \mapsto  \pi(\lambda,z)$  is 
bijective, so $f$ is unbranched at every $y\in Y\setminus D$. The  monodromy 
homomorphism $m=m_{z_{0}}$ satisfies Condition~\eqref{e2.2}, so every 
$[\gamma_{i}]$ acts on the fiber $f^{-1}(y_0)$ as a nontrivial  permutation, 
$i=1,\ldots,n$. Therefore the branch locus of  $f:X\to  Y$  equals  $D$.
We see that $(f:X\to Y,x_0)$ satisfies the conditions of 
Proposition~\ref{2.3}(i) with $(\varepsilon,m)$, so its monodromy invariant is
$(D,m^{N(\lambda_0)})$.
\end{proof}
\begin{proposition}\label{2.7}
The map $[f:X\to Y,x_0]\mapsto (D,m^{N(\lambda_0)})$ establishes a  bijective 
correspondence between the set of equivalence classes of 
pointed $(\Lambda,G)$-covers of $(Y,y_0)$ branched in $n$ points and the set 
$H^{\Lambda,G}_{n,\lambda_0}(Y,y_0)$.
\end{proposition}
\begin{proof}
The map is well-defined and injective by Proposition~\ref{2.3}. It is surjective 
by Lemma~\ref{2.6b}.
\end{proof}
\begin{corollary}\label{2.8a}
Let $(f:X\to Y,x_0)$ be a pointed cover of $(Y,y_0)$ branched in 
$D\in (Y\setminus y_0)_{\ast}^{(n)}$. It is a 
pointed $(\Lambda,G)$-cover of $(Y,y_0)$ and satisfies Condition~(i) of 
Proposition~\ref{2.3} with some pair $(\varepsilon,m)$ if and only if 
it is equivalent to \linebreak
$(\Lambda \times^{G}C,\pi(\lambda_{0},z_{0}))\to (Y,y_0)$ where 
$(p:C\to  Y,z_0)$  is  a  pointed  $G$-cover  of  $(Y,y_0)$  with  monodromy 
invariant $(D,m)$. A pointed $G$-cover of $(Y,y_0)$ has this property if and 
only if its monodromy invariant equals $(D,m_{1})$ where 
$m_{1}=\sigma m \sigma^{-1}$ for some $\sigma \in N(\lambda_0)$. The set of 
$G$-equivalence  classes  of  such  pointed  $G$-covers  of  $(Y,y_0)$   has 
cardinality $|N(\lambda_0)|$.
\end{corollary}
\begin{proof}
This follows from Proposition~\ref{2.3}.
\end{proof}
\begin{proposition}\label{2.9}
Let $D=\{b_1,\ldots,b_n\}\subset Y$. Let $f:X\to Y$ be a  cover  branched  in 
$D$. The following conditions are equivalent:
\begin{enumerate}
\item
There is a point $y_0\in Y\setminus D$, a bijection 
$\varepsilon : \Lambda \to f^{-1}(y_0)$ and an epimorphism 
$m:\pi_1(Y\setminus D,y_0)\to G$ such that 
\[
\varepsilon (\lambda m([\alpha])) = \varepsilon (\lambda )\alpha
\quad \text{for}\quad \forall \lambda \in \Lambda,
\forall [\alpha]\in \pi_1(Y\setminus D,y_0).
\]
\item
The condition of (i) holds for every $y\in Y\setminus D$ for some pair 
\linebreak
$\varepsilon_1 : \Lambda \to f^{-1}(y),\; m_1:\pi_1(Y\setminus D,y)\to G$.
\end{enumerate}
Suppose that $f_{1}:X_{1}\to  Y$  and  $f_{2}:X_{2}\to  Y$  are  two  covers 
branched in $D$. Suppose that $f_{1}$ and $f_{2}$ satisfy Condition~(i) 
respectively for some $y_1\in Y\setminus D$ and a pair 
$(\varepsilon_{1},m_{1})$ and for some $y_2\in Y\setminus D$ and a pair 
$(\varepsilon_{2},m_{2})$. Then $f_{1}$ is equivalent to $f_{2}$ if and only 
if there is a path $\gamma :I\to Y\setminus D$, $\gamma(0)=y_{1}$, 
$\gamma(1)=y_{2}$ and a $\sigma \in N_{S(\Lambda)}(G)$ such that 
\begin{equation*}
m_{2}([\alpha]) = \sigma m_{1}([\gamma \cdot \alpha \cdot \gamma^{-}]) 
\sigma^{-1}
\quad \text{for}\quad \forall [\alpha]\in \pi_1(Y\setminus D,y_2).
\end{equation*}
\end{proposition}
\begin{proof}
(i)\: $\Rightarrow$ (ii): Let $y_1\in Y\setminus D$. Let $\gamma :I\to 
Y\setminus  D$  be  a 
path with $\gamma(0)=y_{0}$, $\gamma(1)=y_{1}$. Let 
$\varepsilon_{1} :\Lambda \to f^{-1}(y_1)$ be the bijection defined by 
$\varepsilon_{1}(\lambda)=\varepsilon(\lambda)\gamma$. Let 
$m_{1}:\pi_1(Y\setminus D,y_1)\to G$ be the  epimorphism  $m_{1}=m^{\gamma}$ 
defined by 
$m_{1}([\beta]) = m([\gamma \cdot \beta \cdot \gamma^{-}])$. Then 
\begin{equation*}
\varepsilon_{1} (\lambda m_{1}([\beta])) = 
\varepsilon (\lambda m([\gamma \cdot \beta \cdot \gamma^{-}]))\gamma
= \varepsilon (\lambda )\gamma \cdot \beta \cdot \gamma^{-} \cdot \gamma 
= \varepsilon_{1} (\lambda )\beta.
\end{equation*}
Given $f_{1}$ and $f_{2}$ let $\gamma :I\to Y\setminus D$  be  a  path  such 
that $\gamma(0)=y_{1}$, $\gamma(1)=y_{2}$. Let 
$\varepsilon'_{1} :\Lambda \to f_{1}^{-1}(y_2)$ and  $m'_{1}=m_{1}^{\gamma}$ 
be the pair obtained from $(\varepsilon_{1},m_{1})$ by  $\gamma$  as  above. 
Then by Lemma~\ref{2.2b} $f_{1}$ is equivalent to $f_{2}$ if and only if 
$m_{2}=\sigma m_{1}^{\gamma} \sigma^{-1}$ for some 
$\sigma \in N_{S(\Lambda)}(G)$.
\end{proof}
\begin{block}\label{2.11}
Let $D=\{b_1,\ldots,b_n\}\subset Y$. Given two points 
$y_{1},y_{2}\in Y\setminus D$ and two homomorphisms 
$m_{1}:\pi_1(Y\setminus D,y_1)\to G$ and 
$m_{2}:\pi_1(Y\setminus D,y_2)\to G$, $G\subset S(\Lambda)$, we write 
$m_{1}\sim_N m_{2}$ if there is a path 
$\gamma :I\to Y\setminus D$ with $\gamma(0)=y_{1}$, $\gamma(1)=y_{2}$ and an 
element $\sigma \in N_{S(\Lambda)}(G)$, such that 
$m_{2}([\alpha])=
\sigma m_{1}([\gamma \cdot \alpha \cdot \gamma^{-}]) \sigma^{-1}$ 
for $\forall [\alpha]\in \pi_1(Y\setminus D,y_2)$. This  is  a  relation  of 
equivalence. Given a homomorphism $m:\pi_1(Y\setminus D,y)\to G$  we  denote
by  $[m]$  its  equivalence  class.  It  is  clear  that  if   $m$   is   an 
epimorphism which satisfies  Condition~\eqref{e2.2}  relative  to  the  base 
point $y$ every other homomorphism of $[m]$ has these properties relative to 
its base point. Suppose a cover $f:X\to Y$ branched in $D$ satisfies 
Condition~(i) of Proposition~\ref{2.9} with $(\varepsilon,m)$ for some 
$y_0\in Y$. Then the set of all possible epimorphisms \linebreak
$m_{1}:\pi_1(Y\setminus D,y)\tto G$, $y\in Y\setminus D$ as in 
Proposition~\ref{2.9}(ii) equals the equivalence class $[m]$ 
(apply Proposition~\ref{2.9} to $f_{1}=f=f_{2}$, $y_{1}=y_{0}$,  $y_{2}=y$). 
We denote by $[m]_f$  the  equivalence  class  with  respect  to  $\sim_{N}$ 
determined by $f:X\to Y$. 
\end{block}
\begin{definition}\label{2.12}
In the setup of \S\ref{2.0} we denote by 
$H_n^{\Lambda,G}(Y)$ the set
\begin{equation*}
\begin{split}
H_n^{\Lambda,G}(Y) =  \{(D,[m])|& \; D\in  Y^{(n)}_{\ast}, \;
\text{where}\; m:\pi_1(Y\setminus D,y)\to G \\
&\text{is surjective and satisfies Condition\eqref{e2.2}}\}. 
\end{split}
\end{equation*}
\end{definition}
\begin{definition}\label{2.12a}
A cover $f:X\to Y$ branched in a set $D$  of  $n$  points,  which  satisfies 
Condition~(i) of Proposition~\ref{2.9} for some $y_0\in  Y\setminus  D$  and 
some pair $(\varepsilon,m)$ is called $(\Lambda,G)$-cover of $Y$ branched in 
$n$ points. The  pair  $(D,[m]_{f})\in  H_n^{\Lambda,G}(Y)$  is  called  the 
monodromy invariant of the cover (cf. \S\ref{2.11}).
\end{definition}
\begin{proposition}\label{2.13}
The  map  $[f:X\to  Y]   \mapsto   (D,[m]_{f})$   establishes   a   bijective 
correspondence   between    the    set    of    equivalence    classes    of 
$(\Lambda,G)$-covers of $Y$ branched in $n$ points and the set 
$H_n^{\Lambda,G}(Y)$. If the centralizer of $G$ in $S(\Lambda)$ is trivial,
$Z_{S(\Lambda)}(G)=\{1\}$, then the covering isomorphism between  every  two 
$(\Lambda,G)$-covers with the same monodromy invariant is unique.
\end{proposition}
\begin{proof}
The map is well-defined and injective  by  Proposition~\ref{2.9}.  We  claim 
that it is surjective. Let $(D,[m])\in H_n^{\Lambda,G}(Y)$
where $m:\pi_1(Y\setminus D,y_0)\tto G$ with $y_0\in Y\setminus D$ satisfies 
Condition\eqref{e2.2}. Let $(p:C\to Y,z_0)$ be a pointed $G$-cover of 
$(Y,y_0)$ with monodromy  invariant  $(D,m)$.  Then  the  cover  $f:X\to  Y$ 
constructed in Lemma~\ref{2.6b} is a $(\Lambda,G)$-cover with monodromy 
invariant $(D,[m])$. The last statement of the proposition was proved in 
Lemma~\ref{2.2b}.
\end{proof}
\begin{proposition}\label{2.14}
Let $f:X\to Y$ be a $(\Lambda,G)$-cover with monodromy invariant 
$(D,[m])\in H_n^{\Lambda,G}(Y)$. Then $f$ has a structure of a 
$(\Lambda,G_{1})$-cover if and only if $G_{1}=\varphi  G  \varphi^{-1}$  for 
some $\varphi \in S(\Lambda)$ and the corresponding monodromy invariant is 
$(D,\varphi [m]\varphi^{-1})\in H_n^{\Lambda,G_{1}}(Y)$.
\end{proposition}
\begin{proof}
Let $y_0\in Y\setminus D$ and let $\varepsilon :\Lambda \overset{\sim}{\lto} 
f^{-1}(y_0)$, 
$m:\pi_1(Y\setminus D,y_0)\tto G$ satisfy 
$\varepsilon (\lambda m([\alpha])) = \varepsilon (\lambda )\alpha$ 
for $\forall \lambda \in \Lambda$ and 
$\forall [\alpha] \in \pi_1(Y\setminus D,y_0)$. Let 
$\varphi \in S(\Lambda)$ and let
$G_{1}=\varphi G \varphi^{-1}$. Let 
$\varepsilon_{1} :\Lambda \to f^{-1}(y_0)$ be the bijection defined by 
$\varepsilon_{1}(\lambda)=\varepsilon (\lambda \varphi)$ and let 
$m_{1} = \varphi m \varphi^{-1} : \pi_1(Y\setminus  D,y_0)\to  G_{1}$.  Then 
for $\forall [\alpha] \in \pi_1(Y\setminus D,y_0)$
\begin{equation*}
\varepsilon_{1}(\lambda m_{1}([\alpha])) =
\varepsilon_{1}(\lambda \varphi m([\alpha])\varphi^{-1}) =
\varepsilon (\lambda \varphi m([\alpha])) = 
\varepsilon (\lambda \varphi)\alpha = \varepsilon_{1}(\lambda)\alpha,
\end{equation*}
hence $f:X\to Y$ is a $(\Lambda,G_{1})$-cover with monodromy invariant 
$(D,\varphi [m]\varphi^{-1})$. Viceversa, let 
$\varepsilon_1 :\Lambda \overset{\sim}{\lto} f^{-1}(y_0)$ and
$m_1:\pi_1(Y\setminus D,y_0)\tto G_1$ satisfy
$\varepsilon_1 (\lambda m_1([\alpha])) = \varepsilon_1 (\lambda )\alpha$ 
for $\forall \lambda \in \Lambda$ and 
$\forall [\alpha] \in \pi_1(Y\setminus D,y_0)$. Let $\varphi \in S(\Lambda)$ 
satisfy $\varepsilon_{1}(\lambda) = \varepsilon(\lambda \varphi)$
for $\forall \lambda \in \Lambda$. Then 
\begin{equation*}
\varepsilon((\lambda \varphi)\varphi^{-1}m_{1}([\alpha])\varphi) =
\varepsilon_{1}(\lambda m_{1}([\alpha])) = 
\varepsilon_{1}(\lambda)\alpha = \varepsilon(\lambda \varphi)\alpha =
\varepsilon((\lambda \varphi)m([\alpha])),
\end{equation*}
so $m_{1}([\alpha]) = \varphi m([\alpha]) \varphi^{-1}$
 for $\forall [\alpha] \in \pi_1(Y\setminus D,y_0)$. This implies 
that $G_{1}=\varphi G \varphi^{-1}$, since $m$ and $m_{1}$ are  epimorphisms 
by assumption.
\end{proof}
\begin{proposition}\label{2.14b}
The following two conditions for a cover $f:X\to Y$ are equivalent:
\begin{enumerate}
\item
$(f:X\to Y,x_0)$ is a pointed cover of $(Y,y_0)$ 
(cf. Definition~\ref{2.1}) and $f:X\to Y$ is a $(\Lambda,G)$-cover.
\item
Let $\lambda_{0}\in \Lambda$. Then $(f:X\to Y,x_0)$ is a 
pointed $(\Lambda,G)$-cover of $(Y,y_0)$ according to Definition~\ref{2.6a}.
\end{enumerate}
\end{proposition}
\begin{proof}
Let $D$ be the branch locus of $f$. Condition~(i) implies that 
$y_{0}\notin D$ and there exists a bijection 
$\varepsilon:\Lambda \to f^{-1}(y_0)$ and an epimorphism 
$m:\pi_1(Y\setminus  D,y_0)\to  G$  such  that 
$\varepsilon(\lambda m([\alpha])) = \varepsilon(\lambda) \alpha$ for 
$\forall \lambda \in \Lambda$ and $\forall [\alpha] \in 
\pi_1(Y\setminus D,y_0)$. The group $G$ acts transitively on 
$\Lambda$. One replaces the pair $(\varepsilon,m)$ with the pair 
$(\varepsilon',m')$, where $\varepsilon'(\lambda) =
 \varepsilon (\lambda g)$, $m'=gmg^{-1}$ for an appropriate $g\in G$, 
so that $(\varepsilon',m')$ satisfies the additional condition 
$\varepsilon'(\lambda_{0}) = x_{0}$.
\end{proof}
\begin{remark}\label{2.14bb}
Choosing a marked element $\lambda_{0}\in \Lambda$ and imposing 
$\varepsilon(\lambda_{0})=x_{0}$ is a normalizing condition for 
pointed degree $d$ covers of $(Y,y_0)$ whose monodromy group equals $G$. 
This serves for defining the monodromy invariant 
$(D,m^{N(\lambda_0)})\in H^{\Lambda,G}_{n,\lambda_0}(Y,y_0)$ so that  
the bijection of Proposition~\ref{2.7} to hold. The choice of another marked element 
$\lambda_{1}\in \Lambda$ is discussed in Proposition~\ref{3.41}.
\end{remark}
\section{Smooth, proper families of covers with a fixed monodromy group} \label{s3}
\begin{definition}\label{2.15}
Let $X$ and $S$ be algebraic varieties.
\begin{enumerate}
\item
A morphism $f:X\to Y\times S$ is called a smooth, proper family of covers of  $Y$ 
branched in $n$ points if: $\pi_{2}\circ f: X\to S$ is a proper,  smooth 
morphism, for every $s\in S$ the fiber $X_{s}$ is irreducible and 
$f_s:X_s\to Y$ is a cover branched in $n$ points. Two families 
of this type $f_{1}:X_{1}\to Y\times S$ and $f_{2}:X_{2}\to Y\times  S$  are 
called equivalent if there exists an  isomorphism  $h:X_{1}\to  X_{2}$  such 
that $f_{1}=f_{2}\circ h$. 
\item
If, furthermore, $f_s:X_s\to Y$ is a $(\Lambda,G)$-cover for  $\forall  s\in 
S$ the morphism $f:X\to Y\times S$ is called a smooth, proper family of 
$(\Lambda,G)$-covers of $Y$ branched in $n$ points.
\end{enumerate}
\end{definition}
Under the assumptions of (i) the morphism $f:X\to Y\times S$ is finite, surjective 
and flat \cite[Proposition~2.6]{K7}. If $S$ is connected, then all covers $f_s:X_s\to Y$
have the same degree $d$, where $d$ is the rank of the locally free sheaf 
$f_{\ast}\mathcal{O}_X$.
It is clear that two families are equivalent if and only if there is an 
$S$-isomorphism $h:X_{1}\to X_{2}$, which is a covering isomorphism over $Y$ 
for $\forall s\in S$, i.e. $f_{1s}=f_{2s}\circ h_{s}$ for $\forall s\in S$.
\begin{definition}\label{2.15a}
Let $y_0\in Y$. 
\begin{enumerate}
\item
A smooth, proper family of pointed covers of $(Y,y_0)$ branched in 
$n$ points is a pair 
$(f:X\to Y\times S,\eta:S\to X)$ 
of morphisms of algebraic varieties 
such that $f:X\to Y\times S$ satisfies 
the conditions of Definition~\ref{2.15}(i), for $\forall s\in S$ the cover 
$f_s:X_s\to Y\times \{s\}$ is unbranched at $(y_{0},s)$, and 
$\eta(s)\in f^{-1}(y_0,s)$. Two families of this type 
$(f_{1}:X_{1}\to Y\times S,\eta_{1})$ and 
$(f_{2}:X_{2}\to Y\times S,\eta_{2})$
are called equivalent 
if there exists an  isomorphism  $h:X_{1}\to  X_{2}$  such 
that $f_{1}=f_{2}\circ h$ and $\eta_{2}=h\circ \eta_{1}$.
\item
If, furthermore, $f_s:X_s\to Y\times \{s\}$ is a $(\Lambda,G)$-cover 
for  $\forall  s\in S$ the pair $(f:X\to Y\times S,\eta:S\to X)$ is called 
a smooth, proper family of pointed $(\Lambda,G)$-covers of $(Y,y_0)$ 
branched in  $n$ points.
\end{enumerate}
\end{definition}
\begin{proposition}\label{2.15b}
Let $y_0\in Y$ and let $\lambda_0 \in \Lambda$. 
Let $(f:X\to Y,\eta :S\to X)$ be a smooth, proper family of  pointed 
$(\Lambda,G)$-covers of $(Y,y_0)$ branched in  $n$ points.
Then for every $s\in S$ the pointed cover ($f_s:X_s\to Y,\eta(s))$
of $(Y,y_0)$ satisfies the conditions of Definition~\ref{2.6a} 
(we identify $Y\times \{s\}$ with $Y$).
\end{proposition}
\begin{proof}
This follows from Proposition~\ref{2.14b}.
\end{proof}
\begin{block}\label{2.16}
Let $D=\{b_1,\ldots,b_n\}\subset Y\setminus y_0$,
let $\overline{U}_{1},\ldots,\overline{U}_{n}$ be disjoint 
embedded closed disks in $Y\setminus y_0$ such  that  $b_{i}\in 
U_{i}$  for  $\forall  i$, where $U_{i}$ is the interior of 
$\overline{U}_{i}$. 
Let  $N_D(U_1,\ldots,U_n)\subset  (Y\setminus  y_0)^{(n)}_{\ast}$   be   the  
open  neighborhood of $D$ in the complex topology (i.e. that of 
$((Y\setminus y_0)^{(n)}_{\ast})^{an}$), which consists 
of $E=\{y_{1},\ldots,y_{n}\}$ such that $y_{i}\in U_{i}$ for every $i$.
The inclusion 
$Y\setminus \cup_{i=1}^{n} U_{i}\hookrightarrow Y\setminus 
D$ is a deformation retract, so  for  every  homomorphism 
$m:\pi_1(Y\setminus D,y_0)\to G$
 and  every  $E\in N_D(U_1,\ldots,U_n)$ there is a unique  homomorphism 
$m(E): \pi_1(Y\setminus E,y_0)\to G$ such that the following diagram commutes:
\begin{equation}\label{e2.16}
\xymatrix{
\pi_1(Y\setminus D,y_0)\ar[rd]_{m}&
\pi_1(Y\setminus \cup_{i=1}^{n} U_i,y_0)\ar[l]_{\cong}\ar[r]^{\cong}\ar[d]&
\pi_1(Y\setminus E,y_0)\ar[dl]^{m(E)}\\
&G
}
\end{equation}
Given a  path $\gamma :I\to Y\setminus E$ 
 we denote by $[\gamma ]_{E}$ its homotopy class in $Y\setminus E$. We denote by 
$N_{(D,m)}(U_1,\ldots,U_n)$ the subset of $H^G_n(Y,y_0)$
\begin{equation}\label{e2.16bis}
N_{(D,m)}(U_1,\ldots,U_n) = \{(E,m(E))|E\in N_D(U_1,\ldots,U_n)\}.
\end{equation}
The set $H^G_n(Y,y_0)$ has a structure of affine algebraic variety, the map 
$\delta : H^G_n(Y,y_0) \to (Y\setminus y_0)^{(n)}_{\ast}$, defined by 
$\delta(D,m)=D$ is an \'{e}tale cover, the sets  $N_{(D,m)}(U_1,\ldots,U_n)$ 
form an open sets basis of the topology  of  $H^G_n(Y,y_0)^{an}$  and  every 
open  set  \linebreak
$N_D(U_1,\ldots,U_n)\subset  (Y\setminus  y_0)^{(n)}_{\ast}$   is 
evenly covered with respect to the topological covering map $|\delta^{an}|$:
\begin{equation}\label{e2.16a}
\delta^{-1}(N_D(U_1,\ldots,U_n)) = \bigsqcup_{m}N_{(D,m)}(U_1,\ldots,U_n)
\end{equation}
(cf. \cite[Section~1]{K1}).
\end{block}
\begin{lemma}\label{2.17}
Let $f:X\to Y\times S$ be a smooth, proper  family  of  covers  of  $Y$ 
branched in $n$ points. Suppose that $f_{s_0}:X_{s_0}\to Y$ is a 
$(\Lambda,G)$-cover for some $s_{0}\in S$. Then there exists a  neighborhood 
$V\subset |S^{an}|$ of $s_{0}$ such that $f_s:X_s\to Y$ is a 
$(\Lambda,G)$-cover for $\forall s\in V$.
\end{lemma}
\begin{proof}
Let $D=\{b_1,\ldots,b_n\}$ be the branch locus of $f_{s_{0}}$, let 
$y_0\in    Y\setminus    D$    and    let 
$\varepsilon   :\Lambda   \to   f_{s_{0}}^{-1}(y_0)$,    $m:\pi_1(Y\setminus 
D,y_0)\tto  G$  satisfy  $\varepsilon  (\lambda  m([\alpha]))  =  \varepsilon 
(\lambda )\alpha$ for $\forall \lambda \in \Lambda$ and 
$\forall [\alpha] \in \pi_1(Y\setminus D,y_0)$. Let $B\subset Y\times S$  be 
the branch locus of $f$. The map $\beta:S\to Y^{(n)}_{\ast}$ defined by 
$\beta(s)=B_{s}$ is a morphism \cite[Proposition~2.6]{K7}, so 
$|\beta^{an}|: |S^{an}|\to |(Y^{(n)}_{\ast})^{an}|$ is a continuous map. 
Let $N_D(U_1,\ldots,U_n)$ be a neighborhood of $D=\beta(s_{0})$ as in 
\S\ref{2.16}. Let $V_{1}$ be a neighborhood of $s_{0}$ in  $|S^{an}|$  such 
that $\beta(V_{1})\subset N_D(U_1,\ldots,U_n)$. The restriction 
$|X^{an}|\setminus f^{-1}(B)\overset{|f^{an}|}{\lto} 
|(Y\times S)^{an}|\setminus B$ is a topological covering map
\cite[Proposition~2.6]{K7} and $(Y\times S)^{an} \cong Y^{an}\times S^{an}$ 
\cite[\S1.2]{SGA1}. The complex space $S^{an}$ is locally connected 
(cf. \cite[Ch.9 \S3 n.1]{GR}), therefore there is an embedded open disk 
$U\subset Y\setminus \cup_{i=1}^{n}\overline{U}_{i}$, $y_{0}\in U$ and
a connected neighborhood $V$ of $s_{0}$, such that $V\subset V_{1}$, 
$U\times V\subset  Y\times  S\setminus  B$  and  $f^{-1}(U\times  V)$  is  a 
disjoint union of connected open sets homeomorphic to  $U\times  V$.  Denote 
these open sets by $W_{\lambda}$, $\lambda \in \Lambda$, where 
$W_{\lambda}\ni \varepsilon(\lambda)$.  For  every  $s\in  V$  we  define  a 
bijection $\varepsilon_{s} :\Lambda \to f^{-1}(y_0,s)$ and an epimorphism 
$m_{s}:\pi_1(Y\setminus B_{s},y_0)\to G$ by
\begin{equation}\label{e2.18a}
\varepsilon_{s}(\lambda) = f^{-1}(y_{0},s)\cap W_{\lambda}, \quad
m_{s} = m(\beta(s))
\end{equation}
(cf. \S\ref{2.16}). We claim that $(\varepsilon_{s},m_{s})$ satisfies 
Condition~(i) of Proposition~\ref{2.9} for $f_{s}:X_{s}\to Y$ (we identify 
$Y\times \{s\}$ with $Y$). Let $\alpha$ be a loop in 
$Y\setminus \cup_{i=1}^nU_i$ based at $y_{0}$. Let  $g=m([\alpha]_{D})$.  We 
claim that $\varepsilon_{s}(\lambda)\alpha =\varepsilon_{s}(\lambda g)$ for 
$\forall \lambda \in \Lambda$ and $\forall s\in V$. Consider the homotopy 
$F:[0,1]\times V\to Y^{an}\times S^{an}\setminus B$ defined by 
$F(t,s) = (\alpha(t),s)$. Let $\lambda \in \Lambda$. The restriction 
$f|_{W_{\lambda}}:W_{\lambda}\to U\times  V$  is  a  homeomorphism.  Let  us 
denote by $\tilde{F}_{\lambda}(0): \{0\}\times V\to 
|X^{an}\setminus f^{-1}(B)|$ the composition 
$(0,s)\overset{F}{\mapsto} (y_{0},s)\mapsto 
(f|_{W_{\lambda}})^{-1}(y_{0},s)  =  \varepsilon_{s}(\lambda)$.   By   the 
covering homotopy property (cf.  \cite[Ch.2  \S3  Th.3]{Spa})  there  is  a 
unique continuous lifting $\tilde{F}_{\lambda}$ of $F$ which extends 
$\tilde{F}_{\lambda}(0)$
\[
\xymatrix{
              &X^{an}\setminus f^{-1}(B)\ar[d]\\
[0,1]\times V\ar[ru]^-{\tilde{F}_{\lambda}}\ar[r]^-{F}&  Y^{an}\times 
S^{an}\setminus B.
}
\]
One has $\tilde{F}_{\lambda}(1,s_{0}) = \varepsilon(\lambda)\alpha 
= \varepsilon(\lambda g)\in W_{\lambda g}$. The image 
$\tilde{F}_{\lambda}(\{1\}\times V)$ is a connected component of 
$f^{-1}(\{y_0\}\times V)$, so 
$\tilde{F}_{\lambda}(\{1\}\times V)\subset W_{\lambda g}$. This implies that 
$\varepsilon_{s}(\lambda)\alpha = \varepsilon_{s}(\lambda g)$ for 
$\forall s\in V$. Let $s\in V$. By \S\ref{2.16} we have 
\[
g = m([\alpha]_{D}) = m(\beta (s))([\alpha]_{\beta(s)}) = 
m_{s}([\alpha]_{B_{s}}). 
\]
Every homotopy class of 
$\pi_1(Y\setminus  B_{s},y_0)$  may  be  represented  by  a  loop   $\alpha$ 
 in $Y\setminus \cup_{i=1}^nU_i$ based at $y_{0}$, so 
\[
\varepsilon_{s} (\lambda m_{s}([\alpha]_{B_{s}})) = 
\varepsilon_{s} (\lambda )\alpha \quad \text{for}\:
\forall \lambda \in \Lambda\; \text{and}\;
\forall [\alpha] \in \pi_1(Y\setminus B_{s},y_0).
\]
\end{proof}
\begin{proposition}\label{2.20}
Let $f:X\to Y\times S$ be a smooth , proper family  of   covers  of  $Y$ 
branched in $n$ points. Suppose that $S$ is connected. Suppose that 
$f_{s_{0}}:X_{s_{0}}\to Y$ is a $(\Lambda,G)$-cover for some $s_{0}\in S$. 
Then $f:X\to Y\times S$ is a smooth, proper family of $(\Lambda,G)$-covers of $Y$.
\end{proposition}
\begin{proof}
Let $d=|\Lambda|$. Every cover $f_{s}:X_{s}\to Y$, $s\in S$, is of degree $d$ 
since $S$ is connected.
We have to prove that $f_{s}:X_{s}\to Y$ is a $(\Lambda,G)$-cover for 
$\forall s\in S$. For every transitive subgroup $H\subset S(\Lambda)$ let 
$S^{H}$ be the set of points $s\in S$ such that $f_{s}:X_{s}\to Y$ is a 
$(\Lambda,H)$-cover. It is clear that $S=\cup_{H}S^{H}$. By Lemma~\ref{2.17} 
$S^{H}$   is   an   open   set   in   $|S^{an}|$   for   every    $H$.    By 
Proposition~\ref{2.14}, given  two  transitive  subgroups  $H$  and  $K$  of 
$S(\Lambda)$, the following alternative holds: if $H$ and $K$ are conjugated 
in $S(\Lambda)$, then $S^{H}=S^{K}$; otherwise $S^{H}\cap S^{K}= \emptyset$. 
The topological space $|S^{an}|$ is connected 
(cf. \cite[Corollaire~2.6]{SGA1}), therefore $S=S^{G}$.
\end{proof}
We denote by $Var_{\mathbb{C}}$ the category of algebraic varieties and by 
$(Sets)$ the category of sets.
\begin{definition}\label{2.22}
Let    $S$     be     an     algebraic     variety.     We     denote     by 
$\mathcal{H}^{\Lambda,G}_{Y,n}(S)$ the set of
equivalence classes $[f:X\to Y\times S]$ of smooth, proper families of 
$(\Lambda,G)$-covers of $Y$ branched in $n$ points, parameterized by $S$ 
(cf. Definition~\ref{2.15}). We denote by 
$\mathcal{H}^{\Lambda,G}_{(Y,y_0),n}(S)$ the  set of
equivalence classes 
$[f:X\to Y\times S,\eta :S\to X]$
of  smooth, proper 
families of pointed $(\Lambda,G)$-covers of $(Y,y_0)$ branched in  
$n$  points parameterized by $S$ 
(cf. Definition~\ref{2.15a})
\end{definition}
\begin{block}\label{2.22a}
Let $u: T\to S$ be a morphism of algebraic varieties. Given  a  smooth, proper 
morphism  $X\to  S$  of  reduced,  separated  schemes  of  finite  type  over 
$\mathbb{C}$, the pullback morphism $X_{T}=X\times_{S}T\to T$ is  smooth  and 
proper. The scheme $X\times_{S}T$ is reduced since $T$ is reduced (cf. 
\cite[p.184]{Ma2}.  Hence  $X\times_{S}T$  is  isomorphic  to   the   closed 
algebraic  subvariety  of  $X\times  T$  whose  set   of   points   is   the 
set-theoretical fiber product 
$X(\mathbb{C})\times_{S(\mathbb{C})}T(\mathbb{C})$.
\par
Let $f:X\to Y\times S$ be a smooth, proper family of covers of $Y$ branched  in  $n$ 
points. Let $u:T\to S$ be a morphism of algebraic varieties. Let 
$X_{T}\to T$ be the pullback  of  $\pi_{2}\circ  f:X\to  S$.  By  the  above 
argument $X_{T} = \{(x,t)|\pi_{2}\circ f(x)=u(t)\} \subset X\times T$. Let 
$f_{T}:X_{T}\to Y\times T$ and $h:X_{T}\to X$ be the morphisms defined 
respectively by $f_{T}(x,t)=(\pi_{1}\circ f(x),t)$ and $h(x,t)=x$.  One  has 
the following commutative diagram of morphisms:
\begin{equation}\label{e2.22a}
\xymatrix{
X_T
\ar[r]^-{f_T}\ar[d]_-{h}&Y\times T\ar[d]^-{id\times u}
\ar[r]^-{p_2}
&T\ar[d]^-{u}\\
X\ar[r]^-{f}&Y\times S\ar[r]^-{\pi_2}&S
}
\end{equation}
in which the composed diagram and the right square are Cartesian.  Therefore 
the left square is Cartesian as well \cite[Proposition~4.16]{GW}. 
\par
Given a smooth, proper family of $(\Lambda,G)$-covers $f:X\to Y\times S$ branched in 
$n$ points the pullback morphism $f_{T}:X_{T}\to Y\times T$ is a smooth, proper 
family of $(\Lambda,G)$-covers of $Y$ branched in $n$ points. In fact 
 for $\forall t\in T$ the cover 
$(f_{T})_{t}:(X_{T})_{t}\to Y$ is equivalent to the cover 
$f_{u(t)}:X_{u(t)}\to Y$, so the  conditions  of  Definition~\ref{2.15}  are 
satisfied. Furthermore  the  pullbacks  of  equivalent  families   are 
equivalent. This defines a moduli functor 
$\mathcal{H}^{\Lambda,G}_{Y,n}: Var_{\mathbb{C}}\to (Sets)$.
\par
Given a smooth, proper family $(f:X\to Y\times S,\eta :S\to X)$ of
pointed $(\Lambda,G)$-covers of $(Y,y_0)$  branched  in  $n$  points  and  a 
morphism $u:T\to S$ one defines the pullback family as 
$(f_{T}:X_{T}\to Y\times T,\eta_{T} :T\to X_{T})$, where $\eta_{T}$  is  the 
morphism defined by $\eta_{T}(t)=(\eta(u(t)),t)$. This is a smooth, proper family of 
pointed $(\Lambda,G)$-covers of $(Y,y_0)$. This defines, as above, a  moduli 
functor 
$\mathcal{H}^{\Lambda,G}_{(Y,y_0),n}: 
Var_{\mathbb{C}}\to (Sets)$.
\end{block}
\begin{block}\label{3.43a}
A morphism $p:\mathcal{C}\to Y\times  S$  of  algebraic  varieties  is  called 
a smooth, proper family of $G$-covers of $Y$ branched in $n$ points if:
\begin{enumerate}
\item
$\pi_{2}\circ  p:\mathcal{C}\to  S$  is  a  proper,  smooth  morphism   with 
irreducible fibers;
\item
$G$ acts on $\mathcal{C}$ on the left by automorphisms, 
$p:\mathcal{C}\to Y\times S$ is $G$-invariant and for $\forall s\in  S$  the 
action of $G$ on $\mathcal{C}_{s}$ is faithful, 
$\overline{p}_{s}:\mathcal{C}_{s}/G \to Y\times \{s\}$ is an isomorphism and 
$p_{s}:\mathcal{C}_{s}\to Y\times \{s\}$ is branched in $n$ points.
\end{enumerate}
Let $y_{0}\in Y$. If furthermore $p_{s}:\mathcal{C}_{s}\to Y\times \{s\}$ is 
unbranched at $(y_{0},s)$ for $\forall s\in S$ and there exists a morphism 
$\zeta :S\to \mathcal{C}$ such that $p\circ \zeta (s)=(y_{0},s)$ for 
$\forall s\in S$, then  $(p:\mathcal{C}\to  Y\times  S,  \zeta)$  is  called 
a smooth, proper family of pointed $G$-covers of $(Y,y_{0})$ branched in $n$ points 
(cf. \cite[Section~3]{K7}).
\end{block}
In the next two propositions we extend the construction of 
Lemma~\ref{2.6b} to families.
\begin{proposition}\label{3.45}
Let $p:\mathcal{C}\to Y\times S$ be a smooth, proper family of $G$-covers of 
$Y$ branched in $n$ points.
\par
Consider the left action of $G$ on $\Lambda \times \mathcal{C}$ defined by 
$g(\lambda,z)=(\lambda g^{-1},gz)$. Let 
$X=(\Lambda \times \mathcal{C})/G := \Lambda \times^G \mathcal{C}$,
let $\pi:\Lambda \times \mathcal{C}\to X$ be the quotient map and let 
$f:X\to Y\times S$ be the map $f(\pi(\lambda,z))=p(z)$. Then 
$f:X\to Y\times S$ is a smooth, proper family of $(\Lambda,G)$-covers of $Y$. The 
branch loci of $p$ and $f$ coincide.
\par
Let $X_{1}=\mathcal{C}/G(\lambda_{0})$, let $\rho:\mathcal{C} \to X_{1}$  be 
the quotient map and let $f_{1}:X_{1}\to Y\times S$ be the map 
$f_{1}(\rho(z))=p(z)$. Then $f_{1}:X_{1}\to Y\times S$ is a smooth, proper family 
of covers of $Y$ equivalent to $f:X\to Y\times S$. 
\end{proposition}
\begin{proof}
The group $G$ acts by covering automorphisms of the finite morphism 
$\Lambda \times \mathcal{C} \to Y\times S$ defined by 
$(\lambda,z)\mapsto p(z)$. Therefore 
$X=(\Lambda \times \mathcal{C})/G$ has a structure of quotient algebraic 
variety and the maps $\pi:\Lambda \times \mathcal{C}\to X$ and 
$f:X\to Y\times S$ are finite morphisms \cite[Ch.III Prop.~19]{Se}.
 The morphism $\pi_{2}\circ f:X\to S$ is a composition 
of a proper and a finite morphism, so it is proper. For every 
$s\in S$ its scheme-theoretical fiber $X_{s}$  
is isomorphic to the quotient 
$(\Lambda \times \mathcal{C}_{s})/G$, since  formation  of  quotients  by  $G$ 
commutes with base change (cf. \cite[Prop. A.7.1.3]{KM}). Therefore $X_{s}$ is smooth 
and irreducible. Let $H=G(\lambda_{0})$. Similar statements hold for 
$X_{1}=\mathcal{C}/H$, $\rho:\mathcal{C}\to X_{1}$ and 
$f_{1}:X_{1}\to Y\times S$. The two families of covers of $Y$,
$f:X\to Y\times S$ and $f_{1}:X_{1}\to Y\times S$
 are equivalent. In fact the morphisms 
$\mathcal{C}\to \Lambda \times^G \mathcal{C} = X$ and 
$\Lambda \times \mathcal{C}\to \mathcal{C}/H = X_{1}$, defined as in 
\eqref{e2.7a}, induce morphisms $X_{1}\to X$ and $X\to X_{1}$ over 
$Y\times S$, given respectively by $\rho(z)\mapsto \pi(\lambda_{0},z)$ and 
$\pi(\lambda,z)=\rho(a_{\lambda}z)$, which are inverse to each other. 
We claim that $\pi_{2}\circ f$ and $\pi_{2}\circ f_{1}$ are flat morphisms. 
Let $x\in X_{1}$, let $\pi_{2}\circ f_{1}(x) = s$ and let $\rho^{-1}(x)$  be 
the scheme-theoretical fiber of $\rho$. The restriction 
$\rho_{s}:\mathcal{C}_{s}\to (X_{1})_{s}$ is a surjective morphism of 
smooth, irreducible, projective curves, so 
$\dim_{\mathbb{C}}H^{0}(\mathcal{O}_{\rho_{s}^{-1}(x)}) = |H| =
|G|/|\Lambda|$. The schemes $\rho^{-1}(x)$ and $\rho_{s}^{-1}(x)$ coincide,
so $\dim_{\mathbb{C}}H^{0}(\mathcal{O}_{\rho^{-1}(x)}) = |G|/|\Lambda|$ for 
every $x\in X_{1}$. This implies that the coherent sheaf 
$\rho_{\ast}\mathcal{O}_{\mathcal{C}}$ is locally free 
(cf. \cite[Ch.~2, \S5, Lemma~1]{M3}), therefore  $\rho:\mathcal{C}\to  X_{1}$ 
is flat. It is moreover faithfully flat since $\rho(\mathcal{C})  =  X_{1}$. 
By hypothesis $\pi_{2}\circ p:\mathcal{C}\to S$ is flat, therefore 
$\pi_{2}\circ f_{1}:X_{1}\to S$ is flat (cf. \cite[p.46]{Ma2}).  As  we  saw 
above the scheme-theoretical fiber $(X_{1})_{s}$ over every (closed) point 
$s\in S$ is smooth. Therefore $\pi_{2}\circ f_{1}:X_{1}\to S$ is smooth at 
every closed point of $X_{1}$ (cf. \cite[Ch.VII, Thm.(1.8)]{A-K}. The  points 
of the scheme $X_{1}$ at which $\pi_{2}\circ f_{1}$ is smooth form  an  open 
subset. This open subset contains every closed point of  $X_{1}$,  therefore 
it coincides with $X_{1}$ (cf. \cite[Prop.~3.35]{GW}). The smoothness of 
$\pi_{2}\circ f_{1}:X_{1}\to S$ implies the smoothness of 
$\pi_{2}\circ f:X\to S$ since there is an $Y\times S$-isomorphism 
between $X$ and $X_{1}$. For every $s\in S$ one has
$X_{s} \cong (\Lambda \times \mathcal{C}_{s})/G$, so 
the morphism $f_{s}:X_{s}\to  Y$, 
is a $(\Lambda,G)$-cover of $Y$ whose branch locus coincides with that of 
$p_{s}:\mathcal{C}_{s}\to Y$ (cf. Lemma~\ref{2.6b}). Therefore 
$f:X\to Y\times S$  is  a  smooth, proper  family  of  $(\Lambda,G)$-covers  of  $Y$ 
branched in $n$ points and it is equivalent to $f_{1}:X_{1}\to Y\times S$ as 
we saw above.
\end{proof}
\begin{proposition}\label{3.48}
Let $y_{0}\in Y$. Let  $(p:\mathcal{C}\to  Y\times  S,\zeta)$  be  a  smooth, proper 
family of pointed $G$-covers of $(Y,y_0)$ branched in $n$ points. 
\par
Let $X=\Lambda \times^{G}\mathcal{C}$, $f:X\to Y\times S$ be as in 
Proposition~\ref{3.45}. Let $\eta:S\to X$ be the morphism defined by 
$\eta(s)=\pi(\lambda_0,\zeta(s))$. Then $(f:X\to Y\times S,\eta)$ is a 
smooth, proper family of pointed $(\Lambda,G)$-covers of $(Y,y_0)$. The branch loci
of $p$ and $f$ coincide. 
For every $s\in S$, if $(D_s,m_s)$ is the monodromy invariant of 
$(p_s:\mathcal{C}_s\to Y,\zeta(s))$, then $(D_s,m_s^{N(\lambda_0)})$ is the monodromy invariant of 
the pointed $(\Lambda,G)$-cover $(f_s:X_s\to Y,\eta(s))$ of $(Y,y_0)$.
\par
Let $X_{1}=\mathcal{C}/G(\lambda_0)$, $f_{1}:X_{1}\to Y\times S$ be as in 
Proposition~\ref{3.45}. Let $\eta_{1}:S\to X_{1}$ be the morphism defined 
by $\eta_{1}=\rho(\zeta(s))$. Then $(f_{1}:X_{1}\to Y\times S,\eta_{1})$  is 
a smooth, proper family of pointed covers of $(Y,y_0)$ equivalent to 
$(f:X\to Y\times S,\eta)$. 
\end{proposition}
\begin{proof}
One has $f(\eta(s))=p(\zeta(s))=(y_{0},s)$, so 
$(f:X\to Y\times S, \eta)$ is a smooth, proper family of pointed 
$(\Lambda,G)$-covers of $(Y,y_0)$. For every $s\in S$ the scheme-theoretical 
fiber $X_s$ is isomorphic to $(\Lambda \times \mathcal{C}_s)/G$, so the stated 
relation between the monodromy invariants holds by Lemma~\ref{2.6b}.
\par
One has $f_{1}(\eta_{1}(s))=p(\zeta(s))=(y_{0},s)$ for $\forall s\in S$. 
The $Y\times S$-isomorphism $X_{1}\to X$ given by 
$\rho(z)\mapsto \pi(\lambda_0,z)$ transforms $\eta_{1}(s)$ in 
$\eta(s)$ for $\forall s\in S$, so the two families of pointed covers of 
$(Y,y_0)$ are equivalent.
\end{proof}
\section{Universal families of pointed 
covers with fixed monodromy group}\label{s4}
\begin{block}\label{3.1}
Given a reduced complex space $X$ and a properly discontinuous group of 
automorphisms $G$  of  $X$  let  $p:X\to  X/G=Z$  be  the  quotient  map  of 
topological spaces. Cartan defined in \cite{Ca} a sheaf  of  complex  valued 
functions $\mathcal{K}$ on $Z$: for every open subset $V$ of $|Z|$, 
$\mathcal{K}(V)=\mathcal{O}_{X}(p^{-1}(V))^{G}$, every stalk 
$\mathcal{K}_{z}$ is a local ring, and he proved in 
\cite[Th\'{e}or\`{e}me~4]{Ca} that the $\mathbb{C}$-ringed space 
$(Z,\mathcal{K})$ is a reduced complex space. Clearly 
$(Z,\mathcal{O}_{Z})$,  $\mathcal{O}_{Z}=\mathcal{K}$,  is  the   categorical 
quotient of $(X,\mathcal{O}_{X})$ in the category of complex spaces.
\end{block}
\begin{proposition}\label{3.1a}
Let $X$ be a normal algebraic variety. Let $G$ be a finite group, which  has 
the property that every orbit is contained in an affine open set. Then 
$X^{an}/G$ is biholomorphic to $(X/G)^{an}$. 
\end{proposition}
\begin{proof}
Consider the composition of morphisms of $\mathbb{C}$-ringed spaces
\[
(X^{an},\mathcal{O}_{X^{an}}) \to (X,\mathcal{O}_{X}) \to 
(X/G,\mathcal{O}_{X/G}).
\]
By the construction of $\mathcal{K}=\mathcal{O}_{Z}$ it factors as 
\[
(X^{an},\mathcal{O}_{X^{an}})\to (Z,\mathcal{O}_{Z})\to 
(X/G,\mathcal{O}_{X/G}).
\]
This induces a holomorphic map $X^{an}/G = Z \to (X/G)^{an}$ 
(cf. \cite[Th\'{e}or\`{e}me~1.1]{SGA1}). The continuous map 
$|X^{an}/G|\to |(X/G)^{an}|$ is a homeomorphism (cf. Lemma~2.5 of \cite{K7}). 
The algebraic variety $X/G$ is normal, so  
$(X/G)^{an}$ is a normal complex space 
(cf. \cite[Proposition~2.1]{SGA1}). 
Normality implies maximality (cf. \cite[\S2.29]{Fi})
therefore  the  holomorphic homeomorphism
$X^{an}/G \to (X/G)^{an}$ is a biholomorphic map \linebreak
\cite[\S2.29]{Fi}.
\end{proof}
\begin{block}\label{3.4}
Let $p:C\to Y$ be a $G$-cover. Let $b\in Y$ be a branch point. Let 
$p^{-1}(b)=\{w_{1},\ldots,w_{r}\}$. There is an embedded open disk 
$V\subset Y, b\in  V$,  such  that  $p^{-1}(V)=\sqcup_{i=1}^{r}W_{i}$  is  a 
disjoint union of connected components $W_{i}$, $w_{i}\in W_{i}$ for 
$i=1,\ldots,r$, and every $p|_{W_{i}}:W_{i}\to V$ is a surjective cyclic 
analytic covering with Galois group 
$G(w_{i})=St_{G}(w_{i})\cong C_{e}\subset \mathbb{C}^{\ast}$, where 
$|G|=er$. Consider the left action of $G$ on $\Lambda \times C$ defined by 
$g(\lambda,z)=(\lambda g^{-1},gz)$. Let 
$X=(\Lambda \times C)/G := \Lambda \times^{G}C$, let 
$\pi:\Lambda \times C\to X$ be the quotient map and let  $f:X\to  Y$  be  the 
map $f(\pi(\lambda,z))=p(z)$. Let $i$ be an integer, $1\leq i\leq r$, let $w=w_{i}$, $W=W_{i}$. It  is 
clear that $f^{-1}(V)\cong (\Lambda \times W)/G(w)$. Let $x=\pi(\lambda,w)$. 
Let $\{\lambda_{1},\ldots,\lambda_{k}\} = \lambda G(w)$. Then 
$(\{\lambda_{1},\ldots,\lambda_{k}\}  \times   W)/G(w)$   is   the   
connected component of $f^{-1}(V)$ which contains $x=\pi(\lambda,w)$. Let 
$G(\lambda,w)=G(w)\cap G(\lambda)$.  Let  $|G(\lambda,w)|=q$.  Then  $e=kq$, 
the map $(\{\lambda\}\times W)/G(\lambda,w)\to 
(\{\lambda_{1},\ldots,\lambda_{k}\}\times W)/G(w)$ is biholomorphic and  the 
composition of holomorphic maps
\[
\begin{split}
W \to (\{\lambda\}\times W)/G(\lambda,w) \overset{\sim}{\lto} 
&(\{\lambda_{1},\ldots,\lambda_{k}\}\times W)/G(w) \eto \\ 
&(\Lambda \times W)/G(w) \to W/G(w) \cong V
\end{split}
\]
has the following form in local coordinates: 
$s\mapsto u=s^{q}\mapsto t=s^{e} = u^{k}$. This shows that
the ramification index of $f:X\to Y$ at the point $x=\pi(\lambda,w)$ 
equals $k=|\lambda G(w)|$.
\par
Let $v_{0}\in V\setminus \{b\}$, $w_{0}\in W$, $p(w_{0})=v_{0}$. Then 
$f^{-1}(v_0)=\{x_{1}\ldots,x_{d}\}$,  where  $x_{i}=\pi(\lambda_{i},w_{0})$. 
If   $\beta:I\to   V\setminus    \{b\}$    is    a    simple    loop    with 
$\beta(0)=\beta(1)=v_{0}$, then  $w_{0}\beta  =  gw_{0}$,  where  $g$  is  a 
generator of $G(w)$. One has $x_{i}\beta = \pi(\lambda_{i},w_{0}\beta ) = 
\pi(\lambda_{i},gw_{0})   =    \pi(\lambda_{i}g,w_{0})$.    Therefore    the 
decomposition of the permutation $x_{i}\mapsto x_{i}\beta$, $i=1,\ldots,d$ 
into  a  product 
of disjoint cycles, by which one determines the indices of the  ramification 
points over $b$, corresponds to the decomposition of $\Lambda$ into a  union 
of $G(w)$-orbits, as discussed above.
\end{block}
\begin{block}\label{3.7}
Let $y_0\in Y$. Let
\begin{equation}\label{e3.7}
(p:\mathcal{C}(y_0)\to Y\times H^G_n(Y,y_0),\zeta:H^G_n(Y,y_0)\to 
\mathcal{C}(y_0))
\end{equation}
 be the universal family of 
pointed $G$-covers of $(Y,y_0)$ branched in $n$ points 
(cf. \cite[Theorem~3.20]{K7}). We recall from \cite[Section~3]{K7} the local analytic
form of $p$ at the points lying over the branch locus 
$B=\{(y,(D,m))|y\in D\}$. 
Let $D=\{b_1,\ldots,b_j,\ldots,b_n\}$, $y_0\in Y\setminus D$. 
Let us choose 
local analytic  coordinates  $s_{i}$ at $b_{i}$, such that $s_{i}(b_{i})=0$, 
$i=1,\ldots,n$.
Let $\epsilon \in \mathbb{R}^+, \epsilon \ll 1$ be such that the 
open   sets   $U_{i}=\{y|s_{i}(y)<\epsilon\}$   have    disjoint    closures 
$\overline{U}_{i}$,     $i=1,\ldots,n$ and $y_{0}\in Y\setminus 
\cup_{i=1}^n\overline{U}_i$.   
Let $\gamma_1,\ldots,\gamma_n$ be closed paths based at $y_0$ as in 
\S\ref{2.2}. Let $p(w) = (b,(D,m))$, where $b=b_j$. 
 Let    $U=U_{j}$    and    let    $V=U\times N_{(D,m)}(U_1,\ldots,U_n)$ (cf. 
\S\ref{2.16}).
For every $v=(y,(D',m(D'))\in V$, where $y\in U$, $D'=\{y_1,\ldots,y_n\}, 
y_i\in U_i$ 
let $t_i(v) = s_i(y_i)$ and let $t(v)=s_j(y)$. Let $G(w)$  be  the  isotropy 
group of $w$, $|G(w)|=e\geq 2$. Then $w$ has a connected neighborhood 
$W\subset |\mathcal{C}(y_{0})^{an}|$ which is $G(w)$-invariant, 
$p(W)=V$ and $gW\cap W=\emptyset$ if $g\in G\setminus G(w)$. Let 
$E\subset \mathbb{C}\times V$ be the analytic subset defined by 
$z^{e}=t-t_{j}$ and let $p_{1}:E\to V$ be the projection map. Then there  is 
a biholomorphic map $\theta:W\to E$, such that $p_{1}\circ \theta = p|_{W}$. 
The composition 
$\rho = (z,t_{1},\ldots,t_{n})\circ \theta : W\to \mathbb{C}^{n+1}$ maps 
$W$ biholomorphically onto an open subset $\Omega \subset \mathbb{C}^{n+1}$. 
There exists a primitive character $\chi:G(w)\to \mathbb{C}^{\ast}$ such that 
$\theta$ and $\rho$ are $G(w)$-equivariant with respect to the action of 
$G(w)$ on $E$ and $\Omega$ defined respectively by $g(z,v)=(\chi(g)z,v)$  and 
$g(z,z_{1},\ldots,z_{n})=(\chi(g)z,z_{1},\ldots,z_{n})$.
\end{block}
We need a simple case of Cartan's theorem \cite[Th\'{e}or\`{e}me~4]{Ca}.
\begin{lemma}\label{3.8a}
Let $H$ be a cyclic group of order $q$, let $\chi:H\to \mathbb{C}^{\ast}$ be 
a primitive character. Let $\Omega$ be an open  subset  of  $\mathbb{C}^{n}$ 
invariant under the action of $H$ on $\mathbb{C}^{n}$ defined by 
$h(z_{1},z_{2},\ldots,z_{n})=(\chi(h)z_{1},z_{2},\ldots,z_{n})$. Let 
$\Omega/H$ be the quotient complex space \cite[Th\'{e}or\`{e}me~4]{Ca} 
and let $p:\Omega \to \Omega/H$ be the quotient holomorphic map. Let 
$\psi:\Omega\to \mathbb{C}^{n}$ be the map 
$\psi(z_{1},z_{2},\ldots,z_{n}) = (z_{1}^{q},z_{2},\ldots,z_{n})$. Then 
$\Omega_{1}=\psi(\Omega)$ is an open subset of  $\mathbb{C}^{n}$  and  there 
exists a biholomorphic map $\mu: \Omega_{1}\to \Omega/H$ such that 
$p=\mu \circ \psi$. 
\end{lemma}
\begin{proof}
The holomorphic map $\psi:\Omega \to \mathbb{C}^{n}$ is open since 
the map $\mathbb{C}\to \mathbb{C}$ given by $z\mapsto z^{q}$ is open. 
Hence $\Omega_{1}=\psi(\Omega)$ is an  open 
subset of $\mathbb{C}^{n}$ and $|\Omega_{1}|$ is homeomorphic to the quotient 
topological space $|\Omega|/H$. In order to prove that 
$(\Omega_{1},\mathcal{O}_{\Omega_{1}})$ is isomorphic to the quotient 
$\mathbb{C}$-ringed space $(\Omega/H,\mathcal{K})$ as defined in 
\cite[\S4]{Ca} it suffices to verify that for every $a\in \Omega$ if 
$H(a)$ is the isotropy group of $a$ and if $b=\psi(a)$, then
\begin{equation}\label{e3.8a}
(\mathcal{O}_{\Omega,a})^{H(a)} = 
\psi^{\sharp}_{a}(\mathcal{O}_{\Omega_{1},b}).
\end{equation}
Let $a=(a_{1},a_{2},\ldots,a_{n})$. If $a_{1}\neq 0$, then $H(a)=\{1\}$, 
$\psi:\Omega\to \Omega_{1}$ is locally biholomorphic at $a$, so 
\eqref{e3.8a}  holds.  Let  $a_{1}=0$.  Then  $H(a)=H$.  Without   loss   of 
generality we may suppose that $a=(0,0,\ldots,0)$. Let the $H$-invariant 
germ $f\in \mathcal{O}_{\Omega,a}$ be represented by a series 
$\sum_{\alpha_{1},\ldots,\alpha_{n}\geq 0}a_{\alpha_{1}\alpha_{2}\cdots 
\alpha_{n}}
z_{1}^{\alpha_{1}}z_{2}^{\alpha_{2}}\cdots z_{n}^{\alpha_{n}}$ which converges 
absolutely in 
the polydisk $|z_{i}|< \varepsilon_{i}$, $i=1,\ldots,n$. One has 
$\alpha_{1} = q\beta_{1}$ for every $\alpha_{1}$. The series 
$\sum_{\beta_{1},\alpha_{2},\ldots,\alpha_{n}\geq 
0}b_{\beta_{1}\alpha_{2}\cdots \alpha_{n}}
y_{1}^{\beta_{1}}z_{2}^{\alpha_{2}}\cdots z_{n}^{\alpha_{n}}$,
 where $b_{\beta_{1}\alpha_{2}\cdots \alpha_{n}}= 
a_{q\beta_{1}\alpha_{2}\cdots  \alpha_{n}}$,  converges  absolutely  in   the 
polydisk $|y_{1}|< \varepsilon_{1}^{q}, 
|z_{i}|< \varepsilon_{i}$, $i=2,\ldots,n$ and represents a germ 
$f_{1}\in \mathcal{O}_{\Omega_{1},b}$ such that 
$f = \psi_{a}^{\sharp}(f_{1})$. Hence \eqref{e3.8a} holds for 
$\forall a \in \Omega$.
\end{proof}
In the next proposition we use the setup and the notation  of  \S\ref{2.0} 
and \S\ref{3.7}.
\begin{proposition}\label{3.8}
Let $y_0 \in Y$. Let
\begin{equation}\label{e3.8}
(\varphi:\mathcal{X}(y_0)\to Y\times H^G_n(Y,y_0), 
\eta: H^G_n(Y,y_0)\to \mathcal{X}(y_0))
\end{equation}
be the smooth, proper family of pointed $(\Lambda,G)$-covers of $(Y,y_0)$ obtained from
\eqref{e3.7} as in Proposition~\ref{3.48}. The variety $\mathcal{X}(y_0)$ is 
smooth. The branch locus of $\varphi$ is $B$. 
For every element $(D,m)\in H^G_n(Y,y_0)$ the fiber 
$(\varphi_{(D,m)}:\mathcal{X}(y_0)_{(D,m)}\to Y, \eta(D,m))$ has 
monodromy invariant $(D,m^{N(\lambda_0)})$.
\par
Let $x =\pi(\lambda,w)\in \mathcal{X}(y_0)$ be a point such that 
$\varphi(x) = p(w)\in B$. Let $\varphi(x)=(b_{j},(D,m))$, where 
$D=\{b_1,\ldots,b_{j},\ldots,b_n\}$. Let $|\lambda G(w)|=k$. Then 
there exists a connected, open neighborhood $A$ of $x$ in 
$|\mathcal{X}(y_0)^{an}|$, such that 
$\varphi(A)=V$, where 
$V = \linebreak
U\times N_{(D,m)}(U_1,\ldots,U_n)$, $U=U_j$ and the following 
properties hold. Let  $E_{1}\subset  \mathbb{C}\times  V$  be  the  analytic 
subset defined by $z^{k}=t-t_{j}$ and let $\varphi_{1}:E_{1}\to V$ be the 
projection map.
\begin{enumerate}
\item
There exists a biholomorphic map $\theta_{1}:A\to E_{1}$ such that 
$\varphi_{1}\circ \theta_{1} = \varphi|_{A}$.
\item
The composition $\rho_{1}=(z,t_{1},\ldots,t_{n})\circ \theta_{1}:
A\to \mathbb{C}^{n+1}$ maps $A$ biholomorphically onto an open subset of 
$\mathbb{C}^{n+1}$.
\end{enumerate}
\end{proposition}
\begin{proof}
One applies Proposition~\ref{3.48} to $(p:\mathcal{C}(y_0)\to Y\times H^G_n(Y,y_0),\zeta)$. 
The composition of the smooth morphisms 
$\mathcal{X}(y_0)\to H^G_n(Y,y_0)\to \Spec \mathbb{C}$ is smooth, so $\mathcal{X}(y_0)$
is a smooth variety.
\par
It remains to prove the statements about the local analytic form of $\varphi$. According to 
Proposition~\ref{3.1a} $\mathcal{X}(y_0)^{an}$ is biholomorphic to 
$(\Lambda \times \mathcal{C}(y_0)^{an})/G$. Let 
$G(\lambda,w)=G(\lambda)\cap G(w)$. This is a cyclic group of order $q$ and
$e=kq$. The open subset $\{\lambda\}\times W\subset
\Lambda \times \mathcal{C}(y_0)$ is $G(\lambda,w)$-invariant and for every 
$g\in G\setminus G(\lambda,w)$ one has $g (\{\lambda\}\times W)\cap
(\{\lambda\}\times W)=\emptyset$. Let $A=\pi(\{\lambda\}\times W)$.
The open complex subspace 
$A\subset \mathcal{X}(y_0)^{an}$ is biholomorphic to 
$\{\lambda\}\times W/G(\lambda,w)$. Let $H=G(\lambda,w)$ and let 
$\chi_{1}:H\to \mathbb{C}^{\ast}$ be the restriction of 
$\chi:G\to \mathbb{C}^{\ast}$ (cf. \S\ref{3.7}). Identifying 
$\{\lambda\}\times W$ with $W$ the biholomorphic map 
$\theta:\{\lambda\}\times W\to E$ is $H$-equivariant. Let 
$\psi_{1}:E\to E_{1}$ be the map $\psi_{1}(z,v)=(z^{q},v)$. The 
composition $\psi_{1}\circ \theta$ is $H$-invariant, so there is a 
holomorphic map $\theta_{1}:A\to E_{1}$ such that 
$\psi_{1}\circ \theta = \theta_{1}\circ \pi|_{\{\lambda\}\times W}$.
Let $\Omega = \rho(\{\lambda\}\times W)$, let 
$\psi:\mathbb{C}^{n+1}\to \mathbb{C}^{n+1}$ be the map 
$\psi(z,z_{1},\ldots,z_{n})=(z^{q},z_{1},\ldots,z_{n})$ and let 
$\Omega_{1}=\psi(\Omega)$. One has the following commutative diagram
 of holomorphic maps:
\begin{equation}\label{e3.10}
\xymatrix{
\{\lambda\}\times W
\ar[r]^-{\theta}\ar[d]_-{\pi}&E\; \ar[d]^-{\psi_{1}}
\ar[r]^-{(z,t_{1},\ldots,t_{n})}
&\; \Omega\ar[d]^-{\psi}\\
A\ar[r]^-{\theta_{1}}&\; \; E_1\; \ar[r]^-{(z,t_{1},\ldots,t_{n})}&\; \; 
\Omega_{1}.
}
\end{equation}
The vertical maps in the right square of  \eqref{e3.10}  are  $H$-invariant, 
the horizontal maps are biholomorphic (cf. \cite[Proposition~3.18]{K7}) and 
by Lemma~\ref{3.8a} $\Omega/H\cong \Omega_{1}$, therefore 
$E_{1}\cong E/H$. This implies that $\theta_{1}:A\to E_{1}$ and 
$\rho_{1}:A\to \Omega_{1}$ are biholomorphic maps. The  map  $p_{1}:E\to  V$ 
equals $\varphi_{1}\circ \psi_{1}$, therefore the equality 
$p_{1}\circ \theta=p|_{\{\lambda\}\times W}$ implies 
$\varphi_{1}\circ \theta_{1} = \varphi|_{A}$. 
\par
We mention that 
(ii) implies the smoothness of $\mathcal{X}(y_0)$, as well as
the smoothness of $\pi_2\circ \varphi:\mathcal{X}(y_0)\to H^G_n(Y,y_0)$
using \cite[Ch.~III Proposition~10.4]{Hart}.
\end{proof}
The next proposition is a particular case of Proposition~\ref{3.48}.
\begin{proposition}\label{3.12}
Let $(p:\mathcal{C}(y_0)\to Y\times H^G_n(Y,y_0),\zeta)$ be as in \S\ref{3.7}. Let 
$\rho: \mathcal{C}(y_{0})\to \mathcal{C}(y_{0})/G(\lambda_0)$ be the 
quotient morphism and let 
$\eta_{1}=\rho \circ \zeta:H^G_n(Y,y_0)\to 
\mathcal{C}(y_{0})/G(\lambda_0)$. Then 
$(\overline{p}:\mathcal{C}(y_{0})/G(\lambda_0)\to Y\times H^G_n(Y,y_0), \eta_{1})$  is  a 
family of pointed covers of $(Y,y_0)$ equivalent to the 
family $(\varphi:\mathcal{X}(y_0)\to Y\times H^G_n(Y,y_0),\eta)$ defined in 
Proposition~\ref{3.8}. 
\end{proposition}
In Proposition~\ref{2.7} we proved that the set 
$H^{\Lambda,G}_{n,\lambda_0}(Y,y_0)$ is bijective to the set of 
equivalence classes of pointed $(\Lambda,G)$-covers of $(Y,y_0)$ 
branched in $n$ points.
\begin{proposition}\label{3.12a}
The left action of $N_{S(\Lambda)}(G)$ on $H^G_n(Y,y_0)$ defined by 
$\sigma \ast (D,m) = (D,\sigma m \sigma^{-1})$ is an action by 
covering automorphisms of the \'{e}tale cover \linebreak
$\delta: H^G_n(Y,y_0)\to (Y\setminus y_0)^{(n)}_{\ast}$, where 
$\delta(D,m)=D$ (cf. \S\ref{2.16}). The subgroup 
$N(\lambda_0)=\{\sigma \in N_{S(\Lambda)}(G)|\lambda_{0}\sigma=
\lambda_{0}\}$ acts freely on $H^G_n(Y,y_0)$ and the 
quotient set 
$H^{\Lambda,G}_{n,\lambda_0}(Y,y_0)=H^G_n(Y,y_0)/N(\lambda_0)$ can be 
endowed with a structure of quotient affine algebraic variety. 
The quotient map $\nu:H^G_n(Y,y_0)\to H^{\Lambda,G}_{n,\lambda_0}(Y,y_0)$ 
and the map $\delta_{1}:H^{\Lambda,G}_{n,\lambda_0}(Y,y_0)\to 
(Y\setminus y_0)^{(n)}_{\ast}$ given by 
$\delta_{1}(D,m^{N(\lambda_0)})=D$ are \'{e}tale covers.
\end{proposition}
\begin{proof}
The map $\delta:H^G_n(Y,y_0)\to (Y\setminus y_0)^{(n)}_{\ast}$ given by 
$\delta(D,m)=D$ is a topological covering map with respect to the  canonical 
complex topologies, every neighborhood $N_D(U_1,\ldots,U_n)$ of 
$D\in (Y\setminus y_0)^{(n)}_{\ast}$ is evenly covered (cf. \eqref{e2.16a}). 
It  is  clear   from   \eqref{e2.16bis}   that   for   every   $\sigma   \in 
N_{S(\Lambda)}(G)$  the  map  $(E,\mu)\mapsto  (E,\sigma  \mu  \sigma^{-1})$ 
transforms $N_{(D,m)}(U_1,\ldots,U_n)$ into 
$N_{(D,\sigma m \sigma^{-1})}(U_1,\ldots,U_n)$, so $N_{S(\Lambda)}(G)$ acts 
on $H^G_n(Y,y_0)$ by covering homeomorphism with respect to  $|\delta^{an}|$. 
By \cite[Corollary~4.5]{K7} this action is by covering automorphisms of the 
\'{e}tale cover $\delta:H^G_n(Y,y_0)\to (Y\setminus y_0)^{(n)}_{\ast}$.  The 
subgroup $N(\lambda_0)$ acts freely on $H^G_n(Y,y_0)$ 
(cf. Proposition~\ref{2.3}), so the quotient set 
$H^{\Lambda,G}_{n,\lambda_0}(Y,y_0)  =  H^G_n(Y,y_0)/N(\lambda_0)$   has   a 
structure of a smooth, affine algebraic variety \cite[Ch.III Prop.18]{Se} and
the quotient map $\nu$ is an \'{e}tale cover. One has 
$\delta = \delta_{1}\circ \nu$, so $\delta_{1}$ is a morphism by the 
universal property of quotients and it is an \'{e}tale cover since 
$\delta$ has this property and $\nu$ is \'{e}tale.
\end{proof}
\begin{block}\label{3.12c}
The quotient morphism 
$\nu:H^G_n(Y,y_0)\to H^{\Lambda,G}_{n,\lambda_0}(Y,y_0)$ yields a 
homeomorphism $|H^G_n(Y,y_0)^{an}|/N(\lambda_0) \overset{\sim}{\lto}
|H^{\Lambda,G}_{n,\lambda_0}(Y,y_0)^{an}|$ (cf. \cite[Lemma~2.5]{K7}), so \linebreak
$|H^{\Lambda,G}_{n,\lambda_0}(Y,y_0)^{an}|$ has a neighborhood basis 
consisting of the open sets 
\begin{equation}\label{e3.12c}
\begin{split}
N_{(D,m^{N(\lambda_0)})}(U_1,\ldots,U_n) &=
\nu(N_{(D,m)}(U_1,\ldots,U_n))\\
&=\{(E,m(E)^{N(\lambda_0)})| E\in N_D(U_1,\ldots,U_n)\}.
\end{split}
\end{equation}
For every $\sigma \in N(\lambda_0)$, $\sigma \neq 1$ one has 
$N_{(D,m)}(U_1,\ldots,U_n)\cap 
N_{(D,\sigma m \sigma^{-1})}(U_1,\ldots,U_n) = \emptyset$, so 
$|\delta_{1}^{an}|$ is a topological covering map and every open subset 
$N_D(U_1,\ldots,U_n)\subset   (Y\setminus   y_0)^{(n)}_{\ast}$   is   evenly 
covered:
\begin{equation}\label{e3.12cc}
\delta_{1}^{-1}(N_D(U_1,\ldots,U_n)) = 
\bigsqcup_{m^{N(\lambda_0)}} N_{(D,m^{N(\lambda_0)})}(U_1,\ldots,U_n).
\end{equation}
\end{block}
\begin{block}\label{3.13}
Let $\mathcal{C}(y_0)'=p^{-1}(Y\times H^G_n(Y,y_0)\setminus B)$.  We  recall 
from \cite[Section~3]{K7} that $\mathcal{C}(y_0)'$ is bijective to the set 
$\{(\Gamma_m[\alpha]_D,D,m)\}$ where $(D,m)\in H^G_n(Y,y_0)$, 
$\Gamma_{m}=\Ker(m:\pi_1(Y\setminus D,y_0)\to G)$, $\alpha$ is a path in 
$Y\setminus D$ with $\alpha(0)=y_{0}$ and $[\alpha]_{D}$ is its homotopy 
class in $Y\setminus D$. The map 
$p'=p|_{\mathcal{C}(y_0)'}:\mathcal{C}(y_0)'\to 
Y\times H^G_n(Y,y_0)\setminus B$, defined by $(\Gamma_m[\alpha]_D,D,m)
\mapsto (\alpha(1),(D,m))$, is a topological Galois covering map 
with respect to the topologies of the associated complex spaces,
with group of Deck transformations isomorphic to $G$, where the action of 
$G$ is defined as follows: if $g\in G$, $g=m([\sigma]_{D})$, then 
$g(\Gamma_m[\alpha]_D,D,m) = (\Gamma_m[\sigma\cdot \alpha]_D,D,m)$.
\par
In the setup of \S\ref{2.0} for every $(D,m)\in H^G_n(Y,y_0)$ let 
$\Gamma_{m,\lambda_{0}}=m^{-1}(G(\lambda_0))\subset 
\pi_1(Y\setminus D,y_0)$. Let $\mathcal{X}(y_0)'=
\varphi^{-1}(Y\times H^G_n(Y,y_0)\setminus B)$. By Proposition~\ref{3.12} 
$\mathcal{X}(y_0)'$ is isomorphic to $\mathcal{C}(y_0)'/G(\lambda_0)$ by a 
covering isomorphism over $Y\times H^G_n(Y,y_0)\setminus B$. The quotient 
$\mathcal{C}(y_0)'/G(\lambda_0)$ is bijective to the set 
$\{(\Gamma_{m,\lambda_0}[\alpha]_D,D,m)\}$ and the quotient morphism 
$\rho':\mathcal{C}(y_0)'\to \mathcal{C}(y_0)'/G(\lambda_0)$ is given by
 $(\Gamma_m[\alpha]_D,D,m)\mapsto (\Gamma_{m,\lambda_0}[\alpha]_D,D,m)$.
One has ${\mathcal{X}(y_0)'}^{an}\cong {\mathcal{C}(y_0)'}^{an}/G(\lambda_0)$ 
by Proposition~\ref{3.1a}. The map 
$\varphi':\mathcal{X}(y_0)'\to Y\times H^G_n(Y,y_0)\setminus B$ may be 
identified with 
\begin{equation}\label{e3.13}
(\Gamma_{m,\lambda_0}[\alpha]_D,D,m)\mapsto (\alpha(1),(D,m))
\end{equation}
and it is a topological covering map with respect to the topologies of 
the associated complex spaces.
\end{block}
\begin{proposition}\label{3.14}
For every $(D,m)\in H^G_n(Y,y_0)$ and every $\sigma  \in  N(\lambda_0)$  one 
has  $\Gamma_{m,\lambda_{0}}=\Gamma_{\sigma   m   \sigma^{-1},\lambda_{0}}$. 
Consider the left action of $N(\lambda_0)$ on  the  set  $\mathcal{X}(y_0)'$ 
defined by 
\[
\sigma (\Gamma_{m,\lambda_0}[\alpha]_D,D,m) = 
(\Gamma_{\sigma m \sigma^{-1},\lambda_0}[\alpha]_D,D,\sigma m \sigma^{-1}).
\]
This is an action by covering automorphisms of the composed \'{e}tale cover 
\begin{equation}\label{e3.14}
\mathcal{X}(y_0)'\to Y\times H^G_n(Y,y_0)\setminus B \to 
Y\times (Y\setminus y_0)^{(n)}_{\ast}\setminus A,
\end{equation}
where $A=\{(y,D)|y\in D\}$, and it can be uniquely extended to a left action 
of $N(\lambda_0)$ on $\mathcal{X}(y_0)$ by  covering  automorphisms  of  the 
composed cover 
\[
\mathcal{X}(y_0)\overset{\varphi}{\lto} Y\times H^G_n(Y,y_0)  
\overset{id\times \delta}{\lto}Y\times (Y\setminus 
y_0)^{(n)}_{\ast},
\]
The morphism $\varphi:\mathcal{X}(y_0)\to Y\times H^G_n(Y,y_0)$ is 
$N(\lambda_0)$-equivariant with respect to the action of $N(\lambda_0)$ on 
$Y\times H^G_n(Y,y_0)$ defined by $\sigma \ast (y,(D,m))=
(y,(D,\sigma m \sigma^{-1}))$ \linebreak
(cf. Proposition~\ref{3.12a}).
\end{proposition}
\begin{proof}
For every $\sigma \in N(\lambda_0)$ one has 
$\sigma G(\lambda_0) \sigma^{-1}=G(\lambda_0)$, so 
$\Gamma_{\sigma m \sigma^{-1},\lambda_{0}}=\Gamma_{m,\lambda_{0}}$. A basis 
of open sets of the topology of ${\mathcal{C}(y_0)'}^{an}$  was  constructed 
in \cite[Section~3]{K7} as follows. Let 
$(y,(D,m))\in Y\times H^G_n(Y,y_0)\setminus B$, $D=\{b_1,\ldots,b_n\}$. Let 
$\overline{U},   \overline{U}_{1},\ldots,\overline{U}_{n}$    be    disjoint 
embedded  closed  disks  in  $Y$   with   interiors   $U,U_{1},\ldots,U_{n}$ 
respectively,  such  that  $y\in  U$,  $\overline{U}_{i}\subset   Y\setminus 
\{y_{0}\}$, $b_{i}\in U_{i}$, $i=1,\ldots,n$. Let 
$\alpha :I\to Y\setminus \cup_{i=1}^{n}\overline{U}_{i}$ be a path such that 
$\alpha(0)=y_{0}$, $\alpha(1)=y$. Let 
\begin{equation}\label{e3.15a}
\begin{split}
N_{(\alpha,D,m)}(U,U_1,\ldots,U_n)=
\{&(\Gamma_{m(E)}[\alpha\cdot \tau]_E,E,m(E))|\\
& E\in N_D(U_1,\ldots,U_n), \tau:I\to U, \tau(0)=y\}.
\end{split}
\end{equation}
The family of these subsets of $\mathcal{C}(y_0)'$ is a basis of open sets 
of the topology of ${\mathcal{C}(y_0)'}^{an}$. Furthermore 
(cf. \S\ref{3.13})
\begin{equation}\label{e3.15}
{p'}^{-1}(U\times N_{(D,m)}(U_1,\ldots,U_n)) =
\bigsqcup_{j=1}^{|G|}N_{(\alpha_{j},D,m)}(U,U_1,\ldots,U_n),
\end{equation}
where $m([\alpha_{j}\cdot \alpha_{i}^{-}])\neq 1$ for $i\neq j$. Let 
$h\in G(\lambda_0)$, $h=m([\eta])$, where $\eta$ is a loop in 
$Y\setminus \cup_{i=1}^n\overline{U}_i$. Then 
$hN_{(\alpha,D,m)}(U,U_1,\ldots,U_n)=
N_{(\eta\cdot\alpha,D,m)}(U,U_1,\ldots,U_n)$. We see that acting on 
$\mathcal{C}(y_0)'$ the group $G(\lambda_0)$ permutes the open sets  on  the 
right-hand side of \eqref{e3.15} and the image of
 $N_{(\alpha,D,m)}(U,U_1,\ldots,U_n)$ in $\mathcal{X}(y_0)$ is the set 
\begin{equation}\label{e3.16}
\begin{split}
\overline{N}_{(\alpha,D,m)}(U,U_1,\ldots,U_n) =
\{&(\Gamma_{m(E),\lambda_{0}}[\alpha\cdot \tau]_E,E,m(E))|\\
&E\in N_D(U_1,\ldots,U_n), \tau:I\to U, \tau(0)=y\}.
\end{split}
\end{equation}
These sets form a basis of open sets of the quotient complex topology of
$\mathcal{X}(y_0)'\cong \mathcal{C}(y_0)'/G(\lambda_0)$ and 
\[
{\varphi'}^{-1}(U\times N_{(D,m)}(U_1,\ldots,U_n)) =
\bigsqcup_{i=1}^{d}\overline{N}_{(\alpha_{i},D,m)}(U,U_1,\ldots,U_n),
\]
where $m([\alpha_{j}\cdot \alpha_{i}^{-}])\notin G(\lambda_0)$ for 
$i\neq j$.
\par
Let $\sigma \in N(\lambda_0)$. Then 
$\sigma\overline{N}_{(\alpha,D,m)}(U,U_1,\ldots,U_n)
=   \overline{N}_{(\alpha,D,\sigma    m    \sigma^{-1})}(U,U_1,\ldots,U_n)$. 
Therefore $N(\lambda_0)$ acts on $\mathcal{X}(y_0)'$ by covering 
homeomorphisms of the composed topological covering
\[
\mathcal{X}(y_0)'\overset{\varphi'}{\lto} Y\times H^G_n(Y,y_0)\setminus B  
\lto Y\times (Y\setminus y_0)^{(n)}_{\ast}\setminus A.
\]
By \cite[Corollary~4.5]{K7} $N(\lambda_0)$ acts  on  $\mathcal{X}(y_0)'$  by 
covering automorphisms of the composed \'{e}tale cover \eqref{e3.14}. It  is 
clear from \eqref{e3.13} that $\varphi'$ is $N(\lambda_0)$-equivariant.
\par
Let $H$ be a connected component of $H^G_n(Y,y_0)$. Let 
$\mathcal{X}(y_0)'_{H}={\varphi'}^{-1}(Y\times H\setminus B)$.  This  is  an 
irreducible algebraic variety, quotient of the irreducible variety \linebreak
${p'}^{-1}(Y\times H\setminus B) = \mathcal{C}(y_0)'_{H}$ 
(cf. \cite[\S3.16]{K7}) and the smooth variety 
$\mathcal{X}(y_0)_{H}= \linebreak
\varphi^{-1}(Y\times H)$ is the normalization of 
$Y\times H$ in the field of rational functions 
$\mathbb{C}(\mathcal{X}(y_0)'_{H})$. Given $\sigma  \in  N(\lambda_0)$,  the 
map $x\mapsto \sigma x$ defines an isomorphism 
$\mathcal{X}(y_0)'_{H}\overset{\sim}{\lto}
\mathcal{X}(y_0)'_{\sigma \ast H}$. Passing to normalizations this 
isomorphism extends in a unique way to an isomorphism 
$\mathcal{X}(y_0)_{H}\overset{\sim}{\lto}\mathcal{X}(y_0)_{\sigma \ast H}$.  This 
defines an action of $N(\lambda_0)$ on $\mathcal{X}(y_0)$ with the  required 
properties.
\end{proof}
\begin{block}\label{3.18}
The  group   $N(\lambda_0)$   acts   on   $\mathcal{X}(y_0)$   by   covering 
automorphisms of the cover 
$\mathcal{X}(y_0)\to Y\times (Y\setminus y_0)^{(n)}_{\ast}$. We denote by 
$\mathcal{X}(y_0,\lambda_0)$ the quotient algebraic variety 
$\mathcal{X}(y_0)/N(\lambda_0)$. The $N(\lambda_0)$-equivariant cover 
$\varphi:\mathcal{X}(y_0)\to Y\times H^G_n(Y,y_0)$ descends to a cover 
$\phi:\mathcal{X}(y_0,\lambda_0)\to
Y\times  H^{\Lambda,G}_{n,\lambda_0}(Y,y_0)$.   One   has   the   following 
commutative diagram:
\begin{equation}\label{e3.18}
\xymatrix{
\mathcal{X}(y_0)\ar[r]^-{\kappa}\ar[d]_-{\varphi}&\mathcal{X}(y_{0},\lambda_0)
\ar[d]^-{\phi}\\
Y\times H^G_n(Y,y_0)\ar[r]^-{id\times \nu}&Y\times 
H^{\Lambda,G}_{n,\lambda_0}(Y,y_0),
}
\end{equation}
where $\kappa$ and $\nu$ are the quotient morphisms with respect to the 
actions of $N(\lambda_0)$. The morphism $\kappa \circ \eta:
H^G_n(Y,y_0)\to \mathcal{X}(y_0,\lambda_0)$ maps $(D,m)$ to 
$(\Gamma_{m,\lambda_0}[c_{y_{0}}]_D,D,m^{N(\lambda_0)})$, where 
$m^{N(\lambda_0)}=\{\sigma m \sigma^{-1}|\sigma \in N(\lambda_0)\}$, so 
$\kappa \circ \eta$ is $N(\lambda_0)$-invariant and can be decomposed as 
$\xi \circ \nu$, where 
$\xi:H^{\Lambda,G}_{n,\lambda_0}(Y,y_0)\to 
\mathcal{X}(y_0,\lambda_0)$ is a morphism with the property that 
$\phi(\xi(D,m^{N(\lambda_0)}))=(y_{0},(D,m^{N(\lambda_0)}))$ for 
$\forall (D,m^{N(\lambda_0)})\in H^{\Lambda,G}_{n,\lambda_0}(Y,y_0)$.
\par
The closed algebraic subset $B\subset Y\times H^G_n(Y,y_0)$ is 
$N(\lambda_0)$-invariant. Let $\mathcal{B}= \linebreak
id\times \nu(B)$ and 
let $\mathcal{X}(y_0,\lambda_0)'=
\mathcal{X}(y_0,\lambda_0)\setminus \phi^{-1}(\mathcal{B})$. Then 
$\mathcal{X}(y_0,\lambda_0)' = \mathcal{X}(y_0)'/N(\lambda_0)$ and
$|{\mathcal{X}(y_0,\lambda_0)'}^{an}| = 
|{\mathcal{X}(y_0)'}^{an}|/N(\lambda_0)$ (cf. \cite[Lemma~2.5]{K7}).
 For every \linebreak
$(y,(D,m^{N(\lambda_0)}))\in 
Y\times H^{\Lambda,G}_{n,\lambda_0}(Y,y_0)\setminus \mathcal{B}$ and 
every $z=(\Gamma_{m,\lambda_0}[\alpha]_D,D,m^{N(\lambda_0)})
\in \mathcal{X}(y_0,\lambda_0)'$ such that 
$\phi(z)=(y,(D,m^{N(\lambda_0)}))$, the sets 
\begin{equation}\label{e3.18a}
\begin{split}
N_{(\alpha,D,m^{N(\lambda_0)})}(U,U_1,\ldots,U_n)
=\{&(\Gamma_{m(E),\lambda_0}[\alpha\cdot \tau]_E,E,m(E)^{N(\lambda_0)})|\\
&E\in N_D(U_1,\ldots,U_n), \tau:I\to U, \tau(0)=y\}
\end{split}
\end{equation}
form a basis of neighborhoods of $z$ in the topology of 
${\mathcal{X}(y_0,\lambda_0)'}^{an}$      (cf.      \eqref{e3.15a})      and 
\eqref{e3.16}).
\end{block}
\begin{theorem}\label{3.19}
The pair
\begin{equation}\label{e3.19}
(\phi:\mathcal{X}(y_0,\lambda_0)\to 
Y\times H^{\Lambda,G}_{n,\lambda_0}(Y,y_0),
\xi:H^{\Lambda,G}_{n,\lambda_0}(Y,y_0)\to \mathcal{X}(y_0,\lambda_0))
\end{equation}
is a smooth, proper family of pointed $(\Lambda,G)$-covers of $(Y,y_0)$ branched 
in $n$ points. The variety $\mathcal{X}(y_0,\lambda_0)$ is 
smooth. The branch locus of $\phi$ is $\mathcal{B}$. 
The fiber of the family over 
$(D,m^{N(\lambda_0)})\in H^{\Lambda,G}_{n,\lambda_0}(Y,y_0)$ is a 
pointed $(\Lambda,G)$-cover of $(Y,y_0)$ with monodromy invariant 
$(D,m^{N(\lambda_0)})$.  Every  pointed  $(\Lambda,G)$-cover  of  $(Y,y_0)$ 
branched in $n$ points is equivalent to a unique fiber of the family 
\eqref{e3.19} by a unique covering isomorphism.
\end{theorem}
\begin{proof}
The morphism $\varphi:\mathcal{X}(y_0)\to 
Y\times H^G_n(Y,y_0)$ is $N(\lambda_0)$-equivariant and $N(\lambda_0)$ acts 
freely on  $H^G_n(Y,y_0)$  (cf.   Proposition~\ref{2.3}). 
Therefore   the   action   of   $N(\lambda_0)$   on   the   smooth   variety 
$\mathcal{X}(y_0)$ is free, which implies that the  quotient 
algebraic variety $\mathcal{X}(y_0,\lambda_0)$ is smooth. The morphisms 
$\kappa$ and $\nu$ of \eqref{e3.18} are \'{e}tale covers, so the composition 
$\pi_2\circ \phi:\mathcal{X}(y_0,\lambda_0)\to H^{\Lambda,G}_{n,\lambda_0}(Y,y_0)$
is smooth, since 
$\mathcal{X}(y_0)\to H^G_n(Y,y_0)$ is smooth by Proposition~\ref{3.8}. Furthermore
$\pi_2\circ \phi$ is proper, since $\phi$ is finite and $\pi_2$ is proper.
Acting by $N(\lambda_0)$ on $\mathcal{X}(y_0)$, 
every $\sigma \in N(\lambda_0)$ transforms the 
pointed $(\Lambda,G)$-cover  
$(\mathcal{X}(y_0)_{(D,m)}\to  Y,\eta(D,m))$  of  $(Y,y_0)$  with  monodromy 
invariant   $(D,m^{N(\lambda_0)})$   (cf. Proposition~\ref{3.8})   into   the 
equivalent pointed cover 
$(\mathcal{X}(y_0)_{(D,\sigma m \sigma^{-1})}\to 
Y,\eta(D,\sigma m \sigma^{-1}))$ of $(Y,y_0)$, therefore the fiber of 
the family \eqref{e3.19} over $(D,m^{N(\lambda_0)})$ is a 
pointed cover of $(Y,y_0)$ equivalent to any of these ones. This proves 
the stated properties of the family \eqref{e3.19}. The last statement of 
the theorem follows from Proposition~\ref{2.7} and \S\ref{2.1a}.
\end{proof}
\begin{remark}\label{3.20a}
The local analytic form of 
$\phi:\mathcal{X}(y_0,\lambda_0)\to 
Y\times H^{\Lambda,G}_{n,\lambda_0}(Y,y_0)$ at the points lying over 
the branch locus of $\phi$ is the same as that of 
$\varphi:\mathcal{X}(y_0)\to Y\times H^G_n(Y,y_0)$. In fact, in 
the commutative diagram \eqref{e3.18} $\varphi$ is 
$N(\lambda_0)$-equivariant and the action of $N(\lambda_0)$ is free. 
Let $x_{1}\in \mathcal{X}(y_0,\lambda_0)$ be a point such 
that $\phi(x_{1})$ belongs to $\mathcal{B}$. Let 
$x\in \mathcal{X}(y_0)$ be a point such that $\kappa(x)=x_{1}$ and let 
$\varphi(x)=(b_{j},(D,m))$, where $D=\{b_1,\ldots,b_j,\ldots,b_n\}$ 
(cf. Proposition~\ref{3.8}). Acting on $Y\times H^G_n(Y,y_0)$ by 
$\sigma \ast (y,(D,m))= (y,(D,\sigma m \sigma^{-1}))$, every
$\sigma\in N(\lambda_0)$ transforms $V=U\times N_{(D,m)}(U_1,\ldots,U_n)$ 
into the open set $\sigma  V = U\times N_{(D,\sigma m 
\sigma^{-1})}(U_1,\ldots,U_n)$ 
and $\sigma_1  V\cap \sigma_2  V = \emptyset$ if $\sigma_1 \neq \sigma_2$.
 Respectively, 
acting on $\mathcal{X}(y_0)$, the group $N(\lambda_0)$ permutes 
$\varphi^{-1}(\sigma  V$), 
$\sigma \in N(\lambda_0)$. If $A$ is the neighborhood of 
$x$ as in Proposition~\ref{3.8} let $A_{1}=\kappa(A)$ and let 
$V_{1}=(id_{Y}\times \nu)(V) = 
U\times   N_{(D,m^{N(\lambda_0)})}(U_1,\ldots,U_n)$.   Then   one   has    a 
commutative diagram of holomorphic maps 
\[
\xymatrix{
A\ar[r]^-{\kappa^{an}}\ar[d]_-{\varphi|_A}&A_1\ar[d]^-{\phi|_{A_1}}\\
V\ar[r]^-{id\times \nu^{an}}&V_1
}
\]
in which the horizontal maps are biholomorphic. Therefore statements analogous 
to (i) and (ii) of Proposition~\ref{3.8} hold for 
$\phi|_{A_{1}}:A_{1}\to V_{1}$.
\end{remark}
\begin{proposition}\label{3.23}
Let ($f:X\to Y\times S,\eta :S\to X$) be a smooth, proper family of 
pointed $(\Lambda,G)$-covers of $(Y,y_0)$ branched in $n$ points. Let 
$B\subset Y\times S$ be the branch locus of $f$. 
For every fiber $(f_s:X_s\to Y,\eta(s))$ let $(\varepsilon_s,m_s)$ be a 
pair which satisfies the conditions of Proposition~\ref{2.3}(i) 
(cf. Proposition~\ref{2.15b}). 
For every $s\in S$ let 
$u(s)=(B_{s},m_{s}^{N(\lambda_0)})\in H^{\Lambda,G}_{n,\lambda_0}(Y,y_0)$ 
be the monodromy invariant of $(f_s:X_s\to Y,\eta(s))$. Then 
$u:S\to H^{\Lambda,G}_{n,\lambda_0}(Y,y_0)$ is a morphism.
\end{proposition}
\begin{proof}
The map $\beta:S\to (Y\setminus y_0)^{(n)}_{\ast}\subset Y^{(n)}$ given 
by $\beta(s)=B_{s}$ is a morphism by \cite[Proposition~2.6]{K7}. The map 
$u$ fits in the following commutative diagram:
\[
\xymatrix{
 &   H^{\Lambda,G}_{n,\lambda_0}(Y,y_0)\ar[d]^-{\delta_{1}}\\
S\ar[ru]^-u\ar[r]_-{\beta}&  (Y\setminus y_0)^{(n)}_{\ast}
}
\]
where $\delta_{1}$ is the \'{e}tale cover defined by 
$\delta_{1}(D,m^{N(\lambda_0)})=D$ (cf. Proposition~\ref{3.12a}). In 
order to prove that $u$ is a morphism it suffices to verify that $u$ is
continuous with respect to the topologies of $S^{an}$ 
and $H^{\Lambda,G}_{n,\lambda_0}(Y,y_0)^{an}$ 
 (cf. \cite[Corollary~4.5]{K7}), i.e. we have to show that for every 
$s\in S$ and every neighborhood $N$ of $u(s)$ the point $s$ is internal 
of $u^{-1}(N)$. Let $s_{0}$ be an arbitrary point of $S$. Let 
$\beta(s_{0})=D=\{b_1,\ldots,b_n\}$. 
Let ($\varepsilon :\Lambda \to f_{s_{0}}^{-1}(y_0)$,
$m:\pi_1(Y\setminus D,y_0)\to G$) be a pair which satisfies 
Condition~(i) of Proposition~\ref{2.3} with 
$\varepsilon(\lambda_{0})=\eta(s_{0})$. One has 
$u(s_{0})= (D,m^{N(\lambda_0)})$. Let 
$N_{(D,m^{N(\lambda_0)})}(U_1,\ldots,U_n)$ be any of the open sets of 
the neighborhood basis of 
$(D,m^{N(\lambda_0)})$ in $|H^{\Lambda,G}_{n,\lambda_0}(Y,y_0)^{an}|$ 
contained in $N$ (cf. \eqref{e3.12c}). 
In the course of the proof of Lemma~\ref{2.17} we showed 
that there exists a connected neighborhood $V\subset |S^{an}|$ of $s_0$
such that $\beta(V)\subset N_D(U_1,\ldots,U_n)$ and we 
constructed for every $s\in V$ a pair
$(\varepsilon'_s:\Lambda \to f^{-1}_s(y_0),m'_s=m(\beta(s)))$ 
as in \eqref{e2.18a}. For $\forall s\in S$ this pair satisfies 
the conditions of Proposition~\ref{2.3}(i). In fact, the first one was 
verified in the proof of Lemma~\ref{2.17}. We claim that the second one
is satisfied as well: $\varepsilon'_s(\lambda_0)=\eta(s)$ for $\forall s \in S$.
This holds since $\{y_{0}\}\times V$ is a connected subset 
of $U\times V$, so $f^{-1}(\{y_{0}\}\times V)$ is a disjoint union of 
$d=|\Lambda|$ connected components
$f^{-1}(\{y_{0}\}\times V)\cap W_{\lambda}$, $\lambda \in \Lambda$.
The map 
$\eta: S\to f^{-1}(\{y_{0}\}\times S)$ is continuous  with  respect  to  the 
canonical complex topologies and 
$\eta(s_{0})=\varepsilon(\lambda_{0})\in W_{\lambda_{0}}$. Therefore
 $\eta(V)$ is the connected component of $f^{-1}(\{y_{0}\times V\})$ contained 
in $W_{\lambda_{0}}$. This shows that 
$\varepsilon_{s}(\lambda_{0})=\eta(s)$ for $\forall s\in V$. 
Using Proposition~\ref{2.3} we conclude that for 
$\forall s\in V$ the monodromy invariant $u(s)$ of 
$(f_s:X_s\to Y,\eta(s))$ equals 
$(\beta(s),m(\beta(s))^{N(\lambda_0)})$, so 
$u(s)\in \nu(N_{(D,m)}(U_1,\ldots,U_n)) =
N_{(D,m^{N(\lambda_0)})}(U_1,\ldots,U_n)$ (cf. \eqref{e3.12c}). Therefore 
$V\subset u^{-1}(N)$.
\end{proof}
\begin{lemma}\label{3.43}
Let $f:X\to S$ and  $g:Z\to S$ be proper, surjective morphisms of  algebraic  varieties. 
Let $h:X\to Z$ be a morphism such that $f =  g\circ  h$.  Suppose  that  for 
every $s\in S$ the induced morphism of the scheme-theoretical fibers \linebreak
$h_{s}:X\otimes_{S}\mathbb{C}(s)\to   Z\otimes_{S}\mathbb{C}(s)$    is    an 
isomorphism. Then $h$ is an isomorphism.
\end{lemma}
\begin{proof}
The map $h$ is bijective. By 
\cite[Proposition~4.6.7(i)]{EGAIII}  every   $s\in   S$   has   an   open 
neighborhood $U$ in $S$ such  that  $h|_{f^{-1}(U)}:f^{-1}(U)\to  g^{-1}(U)$  is  a 
closed embedding. Since $X$ and $Z$ are reduced schemes 
$h|_{f^{-1}(U)}$ is an isomorphism. Therefore $h$ is an isomorphism.
\end{proof}
In the next theorem we assume the setup of \S\ref{2.0}.
\begin{theorem}\label{3.25}
The algebraic variety $H^{\Lambda,G}_{n,\lambda_0}(Y,y_0)$ is a fine 
moduli variety for the moduli functor 
$\mathcal{H}^{\Lambda,G}_{(Y,y_0),n}$ of smooth, proper families of 
pointed $(\Lambda,G)$-covers of $(Y,y_0)$ branched in $n$ points 
(cf. \S\ref{2.22a}). The universal family is (cf. Theorem~\ref{3.19})
\begin{equation}\label{e3.25}
(\phi:\mathcal{X}(y_0,\lambda_0)\to 
Y\times H^{\Lambda,G}_{n,\lambda_0}(Y,y_0),\:
\xi:H^{\Lambda,G}_{n,\lambda_0}(Y,y_0)\to \mathcal{X}(y_0,\lambda_0))
\end{equation}
\end{theorem}
\begin{proof}
Let $[f:X\to Y\times S,\eta :S\to X]\in
\mathcal{H}^{\Lambda,G}_{(Y,y_0),n}(S)$ and let 
$B\subset Y\times S$ be the branch locus of $f$.
For every fiber ($f_s:X_s\to Y,\eta(s))$ let $(\varepsilon_s,m_s)$ be a 
pair which satisfies the conditions of Proposition~\ref{2.3}(i) 
(cf. Proposition~\ref{2.15b}).
Let $u:S\to H^{\Lambda,G}_{n,\lambda_0}(Y,y_0)$, 
$u(s)=(B_{s},m_{s}^{N(\lambda_0)})$ be the morphism of 
Proposition~\ref{3.23}. We want to prove that 
$(f:X\to Y\times S,\eta)$ is equivalent to the pullback by $u$ of the 
family \eqref{e3.25}. This is the unique morphism with this 
property since the monodromy invariant classifies the 
pointed $(\Lambda,G)$-covers of $(Y,y_0)$ up to equivalence.
For every $s\in S$ there exists a unique isomorphism 
$h_{s}:X_{s}\to \mathcal{X}(y_0,\lambda_0)_{u(s)}$ such that 
$\phi_{u(s)}\circ h_{s} = (id_{Y}\times u)\circ f_{s}$ and 
$h_{s}(\eta(s))=\xi(u(s))$. Let $h:X\to \mathcal{X}(y_0,\lambda_0)$  be  the 
map whose restriction on every  $X_{s}$  equals  $h_{s}$.  One  obtains  the 
following commutative diagram of maps:
\begin{equation}\label{e3.26}
\xymatrix{
X\ar[r]^-{h}\ar[d]_-{f}&\mathcal{X}(y_0,\lambda_0)\ar[d]^-{\phi}\\
Y\times S\ar[r]^-{id\times u}&Y\times H^{\Lambda,G}_{n,\lambda_0}(Y,y_0).
}
\end{equation}
We want to prove that $h$ is a morphism and that \eqref{e3.26} is a Cartesian 
diagram. Let 
$\mathcal{B}  \subset  Y\times  H^{\Lambda,G}_{n,\lambda_0}(Y,y_0)$  be  the 
branch locus of $\phi$. One has $B=(id_{Y}\times u)^{-1}(\mathcal{B})$. Let 
$X'=X\setminus f^{-1}(B)$, 
$\mathcal{X}(y_0,\lambda_0)'=\mathcal{X}(y_0,\lambda_0)\setminus 
\phi^{-1}(\mathcal{B})$, $h'=h|_{X'}$, $f'=f|_{X'}$, 
$\phi'=\phi|_{\mathcal{X}(y_0,\lambda_0)'}$. Restricting \eqref{e3.26} to 
the complements of the branch loci one obtains the  
commutative  diagram  of maps
\begin{equation*}
\xymatrix{
X'\ar[r]^-{h'}\ar[d]_-{f'}&\mathcal{X}(y_0,\lambda_0)'\ar[d]^-{\phi'}\\
Y\times S\setminus B\ar[r]^-{(id\times u)'}&
Y\times H^{\Lambda,G}_{n,\lambda_0}(Y,y_0)\setminus \mathcal{B}.
}
\end{equation*}
We claim that $h'$ is continuous with respect to the topologies of 
${X'}^{an}$ and ${\mathcal{X}(y_0,\lambda_0)'}^{an}$. Let $s\in S$ and 
let $\varepsilon_{s} :\Lambda \to f_{s}^{-1}(y_0)$,
$m_{s}:\pi_1(Y\setminus B_{s},y_0)\to G$ be a pair as in 
Proposition~\ref{2.3}(i) with $\varepsilon_{s}(\lambda_0)=\eta(s)$. One has 
\[
(f'_{s})_{\ast}\pi_1(X'_{s},\eta(s)) = m_{s}^{-1}(G(\lambda_0)) =
\Gamma_{m_{s},\lambda_{0}}.
\]
Let $x\in X'_{s}$ and let $\gamma :I\to X'_{s}$ be a path such that 
$\gamma(0)=\eta(s)$, $\gamma(1)=x$. Then 
$f'\circ \gamma :I\to Y\setminus B_{s}$ is a path with initial point  $y_0$. 
Lifting $f'\circ \gamma$ in $\mathcal{X}(y_0,\lambda_0)'_{u(s)}$ with 
initial point $\xi(u(s))=
(\Gamma_{m_{s},\lambda_0}[c_{y_{0}}]_{B_{s}},B_{s},m_{s}^{N(\lambda_0)})$ 
its terminal point is 
\begin{equation}\label{e3.27}
h'(x) = h'_{s}(x) = 
(\Gamma_{m_{s},\lambda_0}[f'\circ 
\gamma]_{B_{s}},B_{s},m_{s}^{N(\lambda_0)}).
\end{equation}
For every $x_{0}\in X'$ and every neighborhood $N$ of $h'(x_{0})$ in 
${\mathcal{X}(y_0,\lambda_0)'}^{an}$  we  have  to  prove  that  $x_{0}$  is 
internal point of ${h'}^{-1}(N)$. Let $f'(x_{0})=(y,s_{0})$, 
$D = B_{s_{0}} = \{b_1,\ldots,b_n\}$, 
$m=m_{s_{0}}:\pi_1(Y\setminus D,y_0)\to G$. Let $\gamma_{0}:I\to X'_{s_{0}}$ 
be a path such that $\gamma_{0}(0)=\eta(s_{0})$, $\gamma_{0}(1)=x_{0}$ and 
let $\alpha = f'_{s_{0}}\circ \gamma_{0} :I\to Y\setminus D$. One has 
$\alpha(0)=y_{0}$, $\alpha(1)=y$ and 
$h'(x_{0})=(\Gamma_{m,\lambda_0}[\alpha]_D,D,m^{N(\lambda_0)})$. Let 
$N_{(\alpha,D,m^{N(\lambda_0)})}(U,U_1,\ldots,U_n)$,
with $y\in U$, $b_i\in U_i$, $i=1,\ldots,n$, and
$\alpha(I)\subset Y\setminus \cup_{i=1}^n\overline{U}_i$, be  a  neighborhood 
of $h'(s_{0})$ as in \eqref{e3.18a} contained in $N$. We want to show that 
there exists a neighborhood $W$ of $x_{0}$ such that 
$h'(W)\subset    N_{(\alpha,D,m^{N(\lambda_0)})}(U,U_1,\ldots,U_n)$.     The 
complex space $S^{an}$ is locally connected \cite[Ch.9 \S3 n.1]{GR}, so one 
may shrink the neighborhood $U$ of $y$ and choose a connected neighborhood $V$ 
of $s_{0}$ such 
that $\beta(V)\subset N_D(U_1,\ldots,U_n)$, $U\times V\subset
Y\times S\setminus B$ and ${f'}^{-1}(U\times V)$ is a disjoint union of 
connected open sets homeomorphic to $U\times V$. Let $W$ be the connected 
component of ${f'}^{-1}(U\times V)$ which contains $x_{0}$. We claim that 
$h'(W)\subset  N_{(\alpha,D,m^{N(\lambda_0)})}(U,U_1,\ldots,U_n)$.  Consider 
the homotopy 
\[
F: [0,1]\times V \to Y\times S \setminus B, \quad F(t,s) = 
(\alpha(t),s).
\]
By the covering homotopy property of topological covering maps 
(cf. \cite[Ch.2 \S3 Th.3]{Spa}) there is a unique continuous lifting of $F$
\[
\xymatrix{
              &X'\ar[d]^-{f'}\\
[0,1]\times V\ar[ru]^-{\tilde{F}}\ar[r]^-{F}&  Y\times S\setminus B
}
\]
such that $\tilde{F}(0,s)=\eta(s)$ for $\forall s\in V$. We have 
$f'\circ \tilde{F}(0,s)=f'(\eta(s))=(y_{0},s) = (\alpha(0),s)$. Let 
$s\in V$. The path $t\mapsto \tilde{F}(t,s)$ is a lifting of 
$\alpha \times \{s\}$ with initial point $\eta(s)$, so 
$f'(\tilde{F}(1,s)) = (\alpha(1),s) = (y,s) \in U\times V$. If $s=s_{0}$, 
then $\tilde{F}(1,s_{0}) = \gamma_{0}(1) = x_{0}$. This implies that 
$\tilde{F}(\{1\}\times V)\subset W$, since $V$ is connected. The map 
$f'|_{W}:W\to U\times V$ is a homeomorphism. Let $x\in W$ and let 
$f'(x) = (z,s)$. We construct a path $\gamma$ in $X'_{s}$ which connects 
$\eta(s)$ with $x$ as follows. The path 
$\tilde{\alpha}_{s}(t)=\tilde{F}(t,s)$, $t\in [0,1]$ has initial point 
$\eta(s)$ and terminal point $\tilde{F}(1,s)=w\in W$ such that 
$f'(w) = (\alpha(1),s) = (y,s)$. Let $\tau:I\to U$ be a path such that 
$\tau(0) = y, \tau(1) = z$. Then 
$\gamma = \tilde{\alpha}_{s}\cdot ((f'|_{W})^{-1}\circ (\tau \times \{s\}))$ 
is a path in $X'_{s}$ which connects $\eta(s)$ with $x$ and 
$f'_{s}\circ \gamma = \alpha \cdot \tau$. Let $\beta(s) = E$. We showed in 
Proposition~\ref{3.23} that 
$m_{s}^{N(\lambda_0)} = m(E)^{N(\lambda_0)}$, so by \eqref{e3.27}
\[
h'(x)=(\Gamma_{m(E),\lambda_{0}}[\alpha \cdot \tau]_E,E,m(E)^{N(\lambda_0)})
\]
This shows that 
$h'(W)\subset N_{(\alpha,D,m^{N(\lambda_0)})}(U,U_1,\ldots,U_n)$, so $x_{0}$ 
is an internal point of ${h'}^{-1}(N)$. The claim that $h'$ is continuous is 
proved.
\par
One applies \cite[Corollary~4.5]{K7} to the commutative diagram
\[
\xymatrix{
 &\mathcal{X}(y_0,\lambda_0)'\ar[d]^-{\phi'}\\
X'\ar[ru]^-{h'}\ar[r]_-{(id_Y\times u)'\circ f'}&  Y\times 
H^{\Lambda,G}_{n,\lambda_0}(Y,y_0)\setminus \mathcal{B}
}
\]
and concludes that $h':X'\to \mathcal{X}(y_0,\lambda_0)'$ is a morphism.
\par
Let $\phi_{S}:\mathcal{X}(y_0,\lambda_0)_{S}\to Y\times S$ be  the  pullback 
of $\phi:\mathcal{X}(y_0,\lambda_0)\to 
Y\times H^{\Lambda,G}_{n,\lambda_0}(Y,y_0)$ by 
$u:S\to H^{\Lambda,G}_{n,\lambda_0}(Y,y_0)$. The composition 
$X'\overset{h'}{\lto}\mathcal{X}(y_0,\lambda_0)'
\eto \mathcal{X}(y_0,\lambda_0)$ yields an $S$-morphism 
$\psi: X'\to \mathcal{X}(y_0,\lambda_0)_{S}$ which fits in the following 
commutative diagram of morphisms:
\[
\xymatrix{
X\ar[dr]_-{f} 
&X'\ar@{_{(}->}[l]\ar[r]^-{\psi}\ar[d]&
\mathcal{X}(y_0,\lambda_0)_{S}\ar[dl]^-{\phi_S}\\
&Y\times S\ar[d]^-{\pi_2}\\
&S
}
\]
The graph $\Gamma$ of $\psi$ is contained in the set-theoretical fiber 
product $X\times_{Y\times S}\mathcal{X}(y_0,\lambda_0)_{S}$, which is 
a Zariski closed subset of $X\times \mathcal{X}(y_0,\lambda_0)_{S}$, so 
it contains  the  closure  $\overline{\Gamma}$.  Therefore  the  projection 
morphism $\overline{\Gamma}\to X$ has finite fibers. Applying 
\cite[Theorem~2]{K2} one concludes that $\psi$ can be extended to an 
$S$-morphism $\tilde{\psi}:X\to \mathcal{X}(y_0,\lambda_0)_{S}$. For 
every $s\in S$ the composition $X_{s}\overset{\tilde{\psi}_{s}}{\lto}
(\mathcal{X}(y_0,\lambda_0)_{S})_{s}\overset{\sim}{\lto}
\mathcal{X}(y_0,\lambda_0)_{u(s)}$ is a morphism whose restriction on 
$X'_{s}$ coincides with $h'_{s}$. Hence this composition equals $h_{s}$. 
This implies that $h$ equals the composition of morphisms 
$X\overset{\tilde{\psi}}{\lto}\mathcal{X}(y_0,\lambda_0)_{S}
\lto \mathcal{X}(y_0,\lambda_0)$, so $h:X\to \mathcal{X}(y_0,\lambda_0)$  is 
a morphism. Furthermore $\tilde{\psi}_{s}$ is an isomorphism for every 
$s\in S$. Applying Lemma~\ref{3.43} to the smooth, proper morphisms
$X\to S$ and $\mathcal{X}(y_0,\lambda_0)_{S} \to S$
one concludes that 
$\tilde{\psi}:X\to \mathcal{X}(y_0,\lambda_0)_{S}$ is  an  isomorphism, 
so Diagram~\eqref{e3.26} is Cartesian.  One has that 
$\phi_{S}\circ \tilde{\psi} = f$ and 
$\tilde{\psi}(\eta(s)) = \psi(\eta(s)) = \xi_{S}(s)$ for $\forall  s\in  S$, 
since $h(\eta(s))=\xi(u(s))$ by the construction of the map $h$. Therefore 
$(f:X\to Y\times S,\eta)$ is equivalent to the pullback of the family 
\eqref{e3.25} by the morphism $u:S\to H^{\Lambda,G}_{n,\lambda_0}(Y,y_0)$.
\end{proof}
A pointed cover $(f:X\to Y,x_0)$ of $(Y,y_0)$ is  a  $(\Lambda,G)$-cover  if 
and only if it is equivalent to 
$(\Lambda \times^{G}C \to Y,\pi(\lambda_0,z_{0}))$ for some pointed 
$G$-cover $(p:C\to Y,z_0)$ of $(Y,y_0)$ (cf. Corollary~\ref{2.8a}). 
In the next proposition, applying Theorem~\ref{3.25}, we extend this to smooth, 
proper families of 
pointed covers of $(Y,y_0)$ (cf. also Proposition~\ref{3.48}). 
\begin{proposition}\label{3.49}
Let $y_{0}\in Y$. Let $(f:X\to Y\times S,\eta:S\to X)$ be a smooth, proper family of 
pointed covers of $(Y,y_0)$ branched in $n$ points. Suppose that 
$S$ is connected and there is a point $s_0\in S$ such that 
$(f_{s_{0}}:X_{s_{0}}\to Y,\eta(s_{0}))$ is a $(\Lambda,G)$-cover 
of $(Y,y_0)$. There is an \'{e}tale Galois cover $\mu:T\to S$ with 
Galois group isomorphic to $N(\lambda_0)$ and a smooth, proper family 
$(p_{T}:\mathcal{C}\to Y\times T,\zeta_{T}:T\to \mathcal{C})$ of 
pointed $G$-covers of $(Y,y_0)$ branched in $n$ points such that the 
pullback by $\mu$, $(f_{T}:X_{T}\to Y\times T,\eta_{T}:T\to X_T)$ 
is equivalent to 
$(\Lambda \times^G \mathcal{C}\to Y\times T, 
\pi(\lambda_0,\zeta_{T}):T\to \Lambda \times^G \mathcal{C})$ 
(cf. Proposition~\ref{3.48}). For every $s\in S$ the fibers over the points of 
$\mu^{-1}(s)$ are the $|N(\lambda_0)|$ pointed $G$-covers of $(Y,y_0)$,
nonequivalent to each other, whose 
quotients by $G(\lambda_0)$ are equivalent to the pointed cover $(X_s\to Y,\eta(s))$ of $(Y,y_0)$
(cf. Corollary~\ref{2.8a}).
\end{proposition}
\begin{proof}
By Proposition~\ref{2.20} $(f:X\to Y\times S,\eta)$ is a smooth, proper family of 
pointed $(\Lambda,G)$-covers of $(Y,y_0)$. Let 
$u:S\to H^{\Lambda,G}_{n,\lambda_0}(Y,y_0)$ and 
$h:X\to \mathcal{X}(y_0,\lambda_0)$ be the morphisms defined in the proof 
of Theorem~\ref{3.25}, which make Diagram~\eqref{e3.26} Cartesian. The 
quotient morphism $\nu: H^G_n(Y,y_0)\to H^{\Lambda,G}_{n,\lambda_0}(Y,y_0)$ 
 is an \'{e}tale Galois cover with Galois group 
$N(\lambda_0)$. Let $T$ be the fiber product, which fits  in  the  Cartesian 
diagram 
\begin{equation}\label{e3.50a}
\xymatrix{
T\ar[r]^-{\tilde{u}}\ar[d]_-{\mu}&H^G_n(Y,y_0)\ar[d]^-{\nu}\\
S\ar[r]^-{u}&H^{\Lambda,G}_{n,\lambda_0}(Y,y_0).
}
\end{equation}
The morphism $\mu:T\to S$ is an \'{e}tale Galois cover with Galois group 
$N(\lambda_0)$ (cf.\cite[Prop. A.7.1.3]{KM}).
Let $(f_{T}:X_{T}\to Y\times T,\eta_{T})$ be the pullback of 
$(f:X\to Y\times S,\eta)$ by $\mu:T \to S$ (cf. \S\ref{2.22a}). One has the 
following commutative diagram of morphisms in which the two squares are 
Cartesian:
\begin{equation}\label{e3.50}
\xymatrix{
X_T\ar[r]^-{h_{\mu}}\ar[d]_-{f_{T}}&X\ar[d]_-{f}
\ar[r]^-{h}&\mathcal{X}(y_0,\lambda_0)\ar[d]^-{\phi}\\
Y\times T\ar[r]^-{id\times \mu}&Y\times S\ar[r]^-{id\times u}&
Y\times H^{\Lambda,G}_{n,\lambda_0}(Y,y_0).
}
\end{equation}
Therefore the composed diagram is Cartesian as well 
\cite[Proposition~4.16]{GW}.
Let \linebreak
$(\varphi:\mathcal{X}(y_0)\to Y\times H^G_n(Y,y_0), 
\theta: H^G_n(Y,y_0)\to \mathcal{X}(y_0))$ be
the family of pointed $(\Lambda,G)$-covers of $(Y,y_0)$ defined in 
Proposition~\ref{3.8}. Let $(f_{1}:X_{1}\to Y\times T,\eta_{1})$ 
be its pullback by $\tilde{u}:T\to H^G_n(Y,y_0)$. One has the following 
commutative diagram of morphisms in which the left square is  Cartesian  and 
the right one is \eqref{e3.18}:
\begin{equation}\label{e3.51}
\xymatrix{
X_1\ar[r]^-{h_1}\ar[d]_-{f_1}&\mathcal{X}(y_0)\ar[d]^-{\varphi}
\ar[r]^-{\kappa}&\mathcal{X}(y_0,\lambda_0)\ar[d]^-{\phi}\\
Y\times T\ar[r]^-{id\times \tilde{u}}&Y\times H^G_n(Y,y_0)\ar[r]^-{id\times \nu}&
Y\times H^{\Lambda,G}_{n,\lambda_0}(Y,y_0).
}
\end{equation}
The right square is also Cartesian since the canonical morphism of 
$\mathcal{X}(y_0)$ into the fiber product is a bijective morphism of  smooth 
algebraic varieties. Therefore the composed diagram is Cartesian. We have 
$u\circ \mu = \nu \circ \tilde{u}$ by \eqref{e3.50a}. Comparing 
\eqref{e3.50} with \eqref{e3.51} we conclude that there is an isomorphism 
$q:X_{1}\to X_T$ such that $f_T\circ q = f_{1}$. We claim that 
$q\circ \eta_{1}(t) = \eta_{T}(t)$ for $\forall t\in T$. It suffices to 
check that $\kappa \circ h_{1} \circ \eta_{1}(t) = 
h\circ h_{\mu}\circ \eta_{T}(t)$. One has 
$\kappa(h_{1}(\eta_{1}(t))) = \kappa(\theta(\tilde{u}(t)))
= \xi(\nu(\tilde{u}(t)))$ and $h(h_{\mu}(\eta_{T}(t)))=
h(\eta(\mu(t)))=\xi(u(\mu(t)))=\xi(\nu(\tilde{u}(t)))$. This 
shows that $q:X_{1}\to X$
defines an equivalence of the families of 
pointed covers of $(Y,y_0)$, \linebreak
$(f_{1}:X_{1}\to Y\times T,\eta_{1})$ and 
$(f_{T}:X_{T}\to Y\times T,\eta_{T})$.
\par
Let $(p_{T}:\mathcal{C}\to Y\times T,\zeta_{T})$ be the pullback 
of $(p:\mathcal{C}(y_0)\to Y\times H^G_n(Y,y_0),\zeta)$ by 
$\tilde{u}:T\to H^G_n(Y,y_0)$ \cite[\S5.2]{K7}. One has that 
$\mathcal{C}$ is reduced and may be identified with the closed subvariety  
$\mathcal{C} = \{(z,t)|\pi_{2}\circ p(z)=\tilde{u}(t)\}$ of $\mathcal{C}(y_0)\times T$. 
The induced action  
of $G$ is defined by $g(z,t)=(gz,t)$ and $\zeta_{T}(t)=
(\zeta(\tilde{u}(t)),t)$ for $\forall t\in T$. Let 
$j:\mathcal{C} \to \mathcal{C}(y_0)$ be the $G$-equivariant morphism 
$j(z,t)=z$. Let $H=G(\lambda_0)$ and let 
$\rho:\mathcal{C}(y_0)\to \mathcal{X}(y_0)$ be the $H$-invariant 
morphism defined by $\rho(z)=\pi(\lambda_0,z)$. Let 
$\rho_{1}:\mathcal{C}\to X_{1}$ be the morphism defined by 
$\rho_{1}(z,t)=(\rho(z),t)$. It is $H$-invariant and fits in the 
following commutative diagram of morphisms:
\begin{equation*}
\xymatrix{
\mathcal{C}\ar[r]^-{\rho_{1}}\ar[d]_-{j}&X_1\ar[d]^-{h_1}\ar[r]^-{f_1}&
Y\times T\ar[d]^-{id\times \tilde{u}}\\
\mathcal{C}(y_0)\ar[r]^-{\rho}&\mathcal{X}(y_0)\ar[r]^-{\varphi}&
Y\times H^G_n(Y,y_0).
}
\end{equation*}
One has that $f_{1}\circ \rho_{1}=p_{T}$ and $\varphi \circ \rho = p$, so 
the composed diagram is Cartesian. The right square is Cartesian,
so the left square is Cartesian as well \cite[Proposition~4.16]{GW}. 
By Proposition~\ref{3.12}
$\rho$ induces an isomorphism 
$\overline{\rho}:\mathcal{C}(y_0)/H\overset{\sim}{\lto} \mathcal{X}(y_0)$. 
The formation of quotients by $H$ commutes  with 
base change \cite[Prop. A.7.1.3]{KM}, so $\rho_{1}$ induces an
isomorphism $\overline{\rho}_{1}:\mathcal{C}/H \to X_{1}$ over 
$Y\times T$. Furthermore $\rho_{1}$ transforms the section 
$\zeta_{T}:T\to \mathcal{C}$ into $\eta_{1}:T\to X_{1}$. In fact, 
for every $t\in T$ one has
\[
\rho_{1}(\zeta_{T}(t)) = \rho_{1}(\zeta(\tilde{u}(t)),t) = 
(\rho(\zeta(\tilde{u}(t))),t) = (\theta(\tilde{u}(t)),t) = \eta_{1}(t).
\]
By Proposition~\ref{3.48} we conclude that 
$(f_{1}:X_{1}\to Y\times T,\eta_{1})$ is equivalent to 
$(\Lambda \times^G \mathcal{C}\to Y\times T,\pi(\lambda_0,\zeta_{T}))$. This 
proves the proposition since 
$(f_{T}:X_{T}\to Y\times T,\eta_{T})$ is 
equivalent to $(f_{1}:X_{1}\to Y\times T,\eta_{1})$ as we saw above.
The last statement of the proposition follows from Diagram~\eqref{e3.50a}.
\end{proof}
\begin{block}\label{3.31}
Let $(f:X\to Y,x_0)$ be a pointed $(\Lambda,G)$-cover of $(Y,y_0)$ 
branched in $D=\{b_1,\ldots,b_n\}\subset Y\setminus \{y_{0}\}$ 
 associated with 
$\varepsilon :\Lambda \to f^{-1}(y_0)$,
$m:\pi_1(Y\setminus D,y_0)\to G$ as in Proposition~\ref{2.3}(i). Let 
$\gamma_{1},\ldots,\gamma_{n}$ be loops based at $y_{0}$ as in 
\S\ref{2.1} and let $g_{i}=m([\gamma_{i}])$, $i=1,\ldots,n$. 
Varying $\overline{U}_{1},\ldots,\overline{U}_{n}$ and 
$\eta_{1},\ldots,\eta_{n}$ one obtains 
$\gamma_{1}',\ldots,\gamma_{n}'$ and $g'_{i}=m([\gamma'_{i}])$ such that 
$g'_{i}$ belongs to the conjugacy class of $g_{i}$ in $G$,
$i=1,\ldots,n$ (cf. \S\ref{2.2}). Furthermore replacing 
$(f:X\to Y,x_0)$ by an equivalent pointed $(\Lambda,G)$-cover 
$(f_{1}:X_{1}\to Y,x'_0)$ of $(Y,y_0)$ results in replacing 
$(g'_{1},\ldots,g'_{n})$ by 
$(\sigma g'_{1}\sigma^{-1},\ldots,\sigma g'_{n}\sigma^{-1})$, where 
$\sigma \in N(\lambda_0)$. 
\end{block}
\begin{definition}\label{3.32}
Let $O_{1},\ldots,O_{k}$ be conjugacy classes of $G$, $O_{i}\neq  O_{j}$  if 
$i\ne j$. Let $\underline{n}=n_1O_1+\cdots+n_k O_k$ be a formal sum, where 
$n_{i}\in \mathbb{N}$. Let  $|\underline{n}|=n_{1}+\cdots+n_{k}=n$.
We  say 
that a pointed cover $(f:X\to Y,x_0)$ of $(Y,y_0)$ 
branched in $n$  points  is a $(\Lambda,G)$-cover of 
branching type $\underline{n}$ if there exists a 
bijection  $\varepsilon  :\Lambda  \to  f^{-1}(y_0)$  and
an epimorphism $m:\pi_1(Y\setminus D,y_0)\to G$,  which satisfy
Condition~(i) of Proposition~\ref{2.3}, such that 
\begin{equation}\label{e3.32}
n_{i}\quad \text{of the elements}\quad m([\gamma_{j}])\quad 
\text{belong to}\quad O_{i}\quad \text{for}\quad i=1,\ldots,k.
\end{equation}
\end{definition}
Notice that the branching type is not uniquely determined 
by the equivalence class of $(f:X\to Y,x_0)$. It specifies that 
$(f:X\to Y,x_0)$ is equivalent to \linebreak
$(\Lambda \times^{G}C\to Y,\pi(\lambda_0,z_{0}))$ for some 
pointed $G$-cover $(p:C\to Y,z_0)$ of $(Y,y_0)$ of branching type $\underline{n}$.
\begin{block}\label{3.32a}
Let $H^G_{\underline{n}}(Y,y_0)$ be the subset of $H^G_n(Y,y_0)$  consisting 
of the elements $(D,m)$ with $m$ satisfying Condition~(\ref{e3.32}).
Every nonempty $H^G_{\underline{n}}(Y,y_0)$ is a union 
of connected components of $H^G_n(Y,y_0)$ and 
$H^G_n(Y,y_0) = \bigsqcup_{|\underline{n}|=n} H^G_{\underline{n}}(Y,y_0)$.
Let $\sigma \in N(\lambda_0)$. Then 
$\sigma \ast H^G_{\underline{n}}(Y,y_0) = 
H^G_{\underline{n}'}(Y,y_0)$, with 
$\underline{n}' = n_1O'_1+\cdots+n_k O'_k$, where $O_{i}'$ is the conjugacy class  
$\sigma O_{i} \sigma^{-1}$ of $G$. Suppose 
$H^G_{\underline{n}}(Y,y_0) \neq \emptyset$. Let 
$H^{\Lambda,G}_{\underline{n},\lambda_0}(Y,y_0) = 
\nu(H^G_{\underline{n}}(Y,y_0))$ (cf. Proposition~\ref{3.12a}). 
It is a union of connected components of 
$H^{\Lambda,G}_{n,\lambda_0}(Y,y_0)$. Let us denote by
$\phi_{\underline{n}}:\mathcal{X}_{\underline{n}}(y_0,\lambda_0)\to 
Y\times H^{\Lambda,G}_{\underline{n},\lambda_0}(Y,y_0)$ the restriction of 
the family $\phi: \mathcal{X}(y_0,\lambda_0)\to
Y\times H^{\Lambda,G}_{n,\lambda_0}(Y,y_0)$ and let 
$\xi_{\underline{n}} : 
H^{\Lambda,G}_{\underline{n},\lambda_0}(Y,y_0)
\to \mathcal{X}_{\underline{n}}(y_0,\lambda_0)$ be the restriction of the 
morphism $\xi:H^{\Lambda,G}_{n,\lambda_0}(Y,y_0)\to
\mathcal{X}(y_0,\lambda_0)$.
\end{block}
\begin{theorem}\label{3.34}
Let $\underline{n}=n_1O_1+\cdots+n_k O_k$, $|\underline{n}|=n$, be as in 
Definition~\ref{3.32}.
Let $(f:X\to Y\times S,\eta:S\to X)$ be a smooth, proper family of 
pointed covers of $(Y,y_0)$ branched in $n$ points. Suppose that 
$S$ is connected and there is a point $s_0\in S$ such that 
$(f_{s_{0}}:X_{s_{0}}\to Y,\eta(s_{0}))$ is a $(\Lambda,G)$-cover 
of branching type $\underline{n}$. Then
\begin{enumerate}
\item
there exists a unique morphism 
$u:S\to H^{\Lambda,G}_{\underline{n},\lambda_0}(Y,y_0)$
such that 
$(f:X\to Y\times S,\eta)$ is equivalent to the pullback
by $u$ of $(\phi_{\underline{n}}:\mathcal{X}_{\underline{n}}(y_0,\lambda_0)\to \linebreak
Y\times H^{\Lambda,G}_{\underline{n},\lambda_0}(Y,y_0),\xi_{\underline{n}})$;
\item
there exists an \'{e}tale cover $\mu:T\to S$ and a smooth, proper family of pointed 
$G$-covers $(p:\mathcal{C}\to Y\times T,\zeta:T\to \mathcal{C})$ of $(Y,y_0)$ 
of branching type $\underline{n}$, such that the pullback by $\mu$,
$(f_T:X_T\to Y\times T,\eta_T)$ is equivalent to 
$(\Lambda \times^{G}\mathcal{C}\to Y\times T,\pi(\lambda_0,\zeta_T))$
(cf. Proposition~\ref{3.48}).
\end{enumerate}
\end{theorem}
\begin{proof}
By Proposition~\ref{2.20} $(f:X\to Y\times S,\eta)$ is a smooth, proper family 
of pointed $(\Lambda,G)$-covers 
of $(Y,y_0)$. Part~(i) follows from Theorem~\ref{3.25} since the morphism 
$u:S\to H^{\Lambda,G}_{n,\lambda_0}(Y,y_0)$ of Proposition~\ref{3.23} has image contained in a connected 
component of $H^{\Lambda,G}_{n,\lambda_0}(Y,y_0)$ which is a connected component of 
$H^{\Lambda,G}_{\underline{n},\lambda_0}(Y,y_0)$.
\par
Part~(ii) is proved similarly to Proposition~\ref{3.49} replacing \eqref{e3.50a} by the Cartesian diagram 
\begin{equation}\label{e3.56}
\xymatrix{
T\ar[r]^-{\tilde{u}}\ar[d]_-{\mu}&H^G_{\underline{n}}(Y,y_0)\ar[d]^-{\nu}\\
S\ar[r]^-{u}&H^{\Lambda,G}_{\underline{n},\lambda_0}(Y,y_0)
}
\end{equation}
and using the universal family of pointed $G$-covers of $(Y,y_0)$ of branching type 
$\underline{n}$, $(p_{\underline{n}}:\mathcal{C}_{\underline{n}}(y_0)\to
Y\times H^G_{\underline{n}}(Y,y_0),\zeta_{\underline{n}}$) (cf. \cite[Theorem~5.8]{K7}).
\end{proof}
Choosing another $\lambda_{1}\in \Lambda$ as a marked element  one obtains a family 
of pointed $(\Lambda,G)$-covers of $(Y,y_0)$ which is isomorphic to 
\eqref{e3.19}.
\begin{proposition}\label{3.41}
Let $\lambda_{1}\in \Lambda$ and let $\lambda_{1}=\lambda_{0} g$, 
$g\in G$. Let 
\begin{equation}\label{e3.41}
(\phi_1:\mathcal{X}(y_0,\lambda_1)\to 
Y\times H^{\Lambda,G}_{n,\lambda_1}(Y,y_0),
\xi_1:H^{\Lambda,G}_{n,\lambda_1}(Y,y_0)\to \mathcal{X}(y_0,\lambda_1))
\end{equation}
be the smooth, proper family of pointed $(\Lambda,G)$-covers of  $(Y,y_0)$  branched 
in $n$ points associated with $\lambda_{1}\in \Lambda$ 
(cf. Theorem~\ref{3.19}). Then there exists an isomorphism 
$h:\mathcal{X}(y_0,\lambda_1)\to \mathcal{X}(y_0,\lambda_0)$ which 
fits in the following diagram:
\begin{equation}\label{e3.41a}
\xymatrix{
\mathcal{X}(y_0,\lambda_1)\ar[r]^-{h}\ar[d]_-{\phi_1}&
\mathcal{X}(y_0,\lambda_0)\ar[d]^-{\phi}\\
Y\times H^{\Lambda,G}_{n,\lambda_1}(Y,y_0)\ar[r]^-{id\times u}&
Y\times H^{\Lambda,G}_{n,\lambda_0}(Y,y_0),
}
\end{equation}
where $u:H^{\Lambda,G}_{n,\lambda_1}(Y,y_0)
\to H^{\Lambda,G}_{n,\lambda_0}(Y,y_0)$ is an isomorphism given by 
$u(D,m_{1}^{N(\lambda_1)})=(D,(gm_{1}g^{-1})^{N(\lambda_0)})$ and 
furthermore $\xi \circ u = h\circ \xi_{1}$. Similar statements hold 
for the universal families of pointed $(\Lambda,G)$-covers of $(Y,y_0)$ with 
a fixed branching type $\underline{n}=n_{1}O_{1}+\cdots +n_{k}O_{k}$.
\end{proposition}
\begin{proof}
Let $(f:X\to Y,x_0)$ be a pointed $(\Lambda,G)$-cover of $(Y,y_0)$ 
branched in $D\in (Y\setminus y_0)^{(n)}_{\ast}$ and let 
$(\varepsilon_{1} :\Lambda \to f^{-1}(y_0),
m_{1}:\pi_1(Y\setminus D,y_0)\to G)$ be a pair such that 
$\varepsilon_{1} (\lambda m_{1}([\alpha])) = 
\varepsilon_{1} (\lambda )\alpha$ for $\forall \lambda \in \Lambda$ 
and $\forall [\alpha]\in \pi_1(Y\setminus D,y_0)$, and 
$\varepsilon_{1}(\lambda_{1})=x_{0}$. Let 
$\varepsilon :\Lambda \to f^{-1}(y_0)$ be the bijection defined by 
$\varepsilon(\lambda)=\varepsilon_{1}(\lambda g)$. Then 
$\varepsilon(\lambda_{0})=x_{0}$ and $m=gm_{1}g^{-1}$ satisfies 
$\varepsilon (\lambda m([\alpha])) = \varepsilon (\lambda )\alpha$ 
for $\forall \lambda \in \Lambda$ 
and $\forall [\alpha]\in \pi_1(Y\setminus D,y_0)$. Applying this to 
the fibers of \eqref{e3.41} we see that for 
$\forall  (D,m_{1}^{N(\lambda_1)})  \in  H^{\Lambda,G}_{n,\lambda_1}(Y,y_0)$ 
the monodromy invariant relative to $\lambda_{0}$ of 
$(\mathcal{X}(y_0,\lambda_1)_{(D,m_{1}^{N(\lambda_1)})}
\to Y, \xi_{1}(D,m_{1}^{N(\lambda_1)})$ is 
$(D,(gm_1g^{-1})^{N(\lambda_0)})$, so by Proposition~\ref{3.23} $u$ is 
a morphism. By Theorem~\ref{3.25} there exists a morphism 
$h:\mathcal{X}(y_0,\lambda_1)\to \mathcal{X}(y_0,\lambda_0)$ which makes 
the the Diagram~\eqref{e3.41a} Cartesian and satisfies 
$h\circ \xi_{1} = \xi \circ u$. Replacing $\lambda_{0}$ with $\lambda_{1}$, 
$\lambda_{0}=\lambda_{1} g^{-1}$, one obtains a morphism 
$u_{1}:H^{\Lambda,G}_{n,\lambda_0}(Y,y_0)\to
H^{\Lambda,G}_{n,\lambda_1}(Y,y_0)$ defined by 
$u_{1}(D,m^{N(\lambda_0)})=(D,(g^{-1}mg)^{N(\lambda_1)})$, inverse 
to $u$ and a morphism 
$h_{1}:\mathcal{X}(y_0,\lambda_0)\to \mathcal{X}(y_0,\lambda_1)$. 
One has $h_{1}\circ h = id_{\mathcal{X}(y_0,\lambda_1)}$ since 
$h_{1}\circ h \circ \xi_{1} = \xi_{1}$ and similarly 
$h\circ h_{1} = id_{\mathcal{X}(y_0,\lambda_0)}$.
\end{proof}

\end{document}